\documentclass[12pt,a4paper]{amsart}
\usepackage{amssymb,amscd,amsthm,amsmath}
\usepackage{amsfonts}
\usepackage[top=35mm, bottom=35mm, left=30mm, right=30mm]{geometry}
\usepackage[colorlinks=true,citecolor=blue]{hyperref}
\usepackage{mathptmx}
\usepackage{eucal}
\usepackage{graphicx}
\usepackage{mathrsfs}
\usepackage{xcolor}
\usepackage[pagewise]{lineno}\nolinenumbers

\newtheorem{thm}{Theorem}[section]

\newtheorem{lem}[thm]{Lemma}
\newtheorem{prop}[thm]{Proposition}

\theoremstyle{definition}
\newtheorem{df}[thm]{Definition}
\newtheorem{ex}[thm]{Example}
\newtheorem{rem}[thm]{Remark}
\newtheorem*{thma}{Theorem A}
\newtheorem*{thmb}{Theorem B}
\newtheorem*{thmc}{Theorem C}
\newtheorem*{thmd}{Theorem D}
\numberwithin{equation}{section}
\def\UU{\mathcal{U}}
\def\VV{\mathcal{V}}
\def\diam{\text{\rm diam}}

\def\Leb{\text{\rm Leb}}

\def\logf{\log\frac{1}{\epsilon}}
\def\P{\mathbb{P}}

\def\R{\mathbb{R}}

\def\PP{\mathcal{P}_{\mathbb{P}}(\Omega \times X)}
\def\M{\mathcal{M}_{\mathbb{P}}(\Omega \times X, G)}

\def\FF{\mathcal{G}}

\makeatletter

\newcommand{\Rmnum}[1]{\expandafter\@slowromancap\romannumeral #1@}
\makeatother

\begin{document}

\title{Variational principle for  random pressure function}

\author{Rui Yang$^{1,2}$, Ercai Chen$^{3}$ and Xiaoyao Zhou$^{*3}$
}
\address
{1.College of Mathematics and Statistics, Chongqing University, Chongqing 401331, P.R.China}
\address{2. Key Laboratory of Nonlinear Analysis and its Applications (Chongqing University), Ministry of Education}
\address
{3.School of Mathematical Sciences and Institute of Mathematics, Ministry of Education Key Laboratory of NSLSCS, Nanjing Normal University, Nanjing, 210023, Jiangsu, P.R.China}

\email{zkyangrui2015@163.com}
\email{ecchen@njnu.edu.cn}
\email{zhouxiaoyaodeyouxian@126.com}

\renewcommand{\thefootnote}{}

\footnotetext{*Corresponding author}


 
\renewcommand{\thefootnote}{\arabic{footnote}}
\setcounter{footnote}{0}

\begin{abstract}
For random dynamical systems, by summarizing the fundamental properties of Kifer's topological pressure we introduce the concept of random pressure functions, and define Ruelle's metric entropy for invariant measures. Employing the techniques from convex analysis and ergodic theory, we establish a variational principle for random pressure functions. Consequently, this new variational principle allows us to establish a vital bridge between ergodic theory and topological dynamics. In particular, the variational principles for polynomial topological entropy in zero entropy systems, mean dimensions in infinite entropy systems, and  preimage entropy-like quantities in non-invertible dynamical systems are obtained.
\end{abstract}


\maketitle
\pagestyle{plain}  
\section{Introduction}

People are always  concerned about the topological complexity of abstract dynamical systems. How to understand the dynamics of such systems is a longstanding theme in dynamical systems. This task is not easy to accomplish, as abstract dynamical systems usually possess certain structures, such as group structures, geometric structures, and measure-preserving structures. A quantitative approach to capturing the complexity of dynamical systems involves introducing entropy-like quantities in both topological and measure-theoretic settings, and establishing the corresponding variational principle for these quantities.  Then the tools from ergodic theory can be  invoked to study the dynamical behaviors of dynamical systems by the certain variational principles.   The classical variational principle (cf.\cite{w82}) asserts that the topological entropy equals the supremum of the measure-theoretic entropy over the set of invariant measures.

 Inspired by statistical mechanics, Ruelle and Walters \cite{rue73, wal75} extended the concept of topological entropy to the topological pressure of continuous potentials.  Given a continuous map  $T:X\rightarrow X$   on a compact metric space $(X,d)$ and a continuous potential $f: X\rightarrow \mathbb{R}$, the classical variational principle remains valid for topological pressure \cite{wal75}:
$$P(T,f)=\sup_{\mu \in \mathcal{M}(X,T) }\{h_{\mu}(T)+\int fd\mu\},$$ 
where  $P(T,f)$ denotes  the topological pressure of $f$, and $h_{\mu}(T)$  denotes  the measure-theoretic entropy of the  $T$-invariant  Borel probability measure $\mu$ on $X$.  The variational principles for topological pressure and topological entropy provide a fundamental bridge between topological dynamical systems and ergodic theory. It turns out that the variational principle has become a powerful tool in investigating equilibrium measure theory, thermodynamic formalism, chaotic phenomena in differential dynamical systems, and dimensional estimation in the multifractal analysis of dynamical systems. In real-world scenarios, dynamical systems are usually affected by random external perturbations. Based on families of random transformations and their iterates, the basic framework of random dynamical systems was established by Ulam and von Neumann \cite{uv45}, and further developed by Kakutani \cite{k50} in pursuing the random ergodic theorems. More precisely,  a \emph{continuous  random dynamical system on $X$}  over a measure-preserving system $(\Omega, \mathcal{F},\mathbb{P},\theta)$ is a map $T:G \times \Omega \times X \rightarrow X $ which is  measurable in $\omega \in \Omega$ and  is continuous in $x\in X$, and  for each fixed $\omega$, the set of maps $\{T(n,\omega)\}_{n\in G}$ forms a cocycle over $\theta$. Later, the authors \cite{bog92, kif01} introduced the concept of topological pressure for random dynamical systems, and  proved that the classical variational principle holds for Kifer's topological pressure.

According to the values of topological entropy, dynamical systems can be classified into zero entropy systems, finite positive entropy systems, and infinite entropy systems. The larger the entropy value of a system, the more complicated the system is. Topological polynomial entropy has been introduced in \cite{kt97} as a tool to distinguish zero entropy systems, and further reveals the dynamical behaviors hidden in zero entropy systems.  In \cite[Section 8]{GJ16}, Gr\"{o}ger and J\"{a}ger  posed  the following question:

\emph{Question 1: Does the classical variational principle hold  for polynomial topological entropy?}

Due to the absence of a corresponding measure-theoretic entropy counterpart to polynomial topological entropy, this question is still open for zero entropy systems. Different from zero entropy systems, infinite entropy systems exhibit highly complicated dynamics, yet little extra information about the systems can be obtained from their infinite entropy alone. Mean dimension \cite{gromov} and metric mean dimension \cite{lw00} are two key quantities for understanding the geometric and topological structures of infinite entropy systems. This naturally asks whether the classical variational principle holds for mean dimensions. There has been no progress on this question until Lindenstrauss and Tsukamoto's pioneering work appeared in 2018. In \cite{lt18}, they introduced the rate-distortion function from information theory, and established a variational principle for the (upper) metric mean dimension:
\begin{align*}
{ \rm \overline{mdim}}_M(T,X,d)&=\limsup_{\epsilon \to 0}\frac{1}{\logf}\sup_{\mu \in \mathcal{M}(X,T)}R_{\mu, L^{\infty}}(\epsilon)
\end{align*}
where ${\rm \overline{mdim}}_M(T,X,d)$ denotes the upper metric mean dimension of $X$,  and $R_{\mu, L^{\infty}}(\epsilon)$ is the $L^{\infty}$-rate distortion function of $\mu$. Comparing the Lindenstrauss-Tsukamoto variational principle with the classical ones, unfortunately, for certain dynamical systems \cite[Section VIII]{lt18}, the strict inequality $\sup_{\mu \in \mathcal{M}(X,T)}\{\limsup_{\epsilon \to 0}\frac{1}{\logf}R_{\mu, L^{\infty}}(\epsilon)\}<{\rm \overline{mdim}}_M(T,X,d)$
can happen for $L^{\infty}$ rate-distortion functions if we exchange the order of $\sup_{\mu \in \mathcal{M}(X,T)}$ and $\limsup_{\epsilon \to 0}$. Later, for systems with the marker property,  Lindenstrauss and Weiss  \cite{lt19} further established a  double variational principle for mean dimension:
\begin{align*}
{\rm mdim}(X,T)
=\min_{d\in \mathscr{D}(X)} \sup_{\mu \in \mathcal{M}(X,T)}\{\limsup_{\epsilon \to 0}\frac{1}{\logf}R_{\mu, L^{1}}(\epsilon)\},
\end{align*}
where ${\rm mdim}(X,T)$ denotes  the mean dimension of $X$,  $R_{\mu, L^{1}}(\epsilon)$ is the $L^{1}$-rate distortion function of $\mu$, and  $\mathscr{D}(X)$ is the set of all compatible metrics on $X$. The  marker property implies  aperiodicity. However, there exist periodic dynamical systems without the marker property for which the double variational principle still holds.  Therefore, the variational principles for mean dimensions are far from being established for dynamical systems with infinite entropy. Beyond this, the mean dimensions of amenable groups have been studied beyond $\mathbb{Z}$-actions \cite{cdz22,glt16}, and it is well-known that the classical variational principle holds for amenable group actions \cite{op82, st80}. It is natural to ask the following question:

\emph{Question 2: How to establish the classical variational principle  for mean dimensions of amenable groups without additional assumption?}

Notice that for any continuous self-map $T$ on $X$, the cardinality of  the preimage  $T^{-n}x$ of $x$ is  generally  not  a singleton and may even be uncountable. This means that the backward orbits  exhibit   rich preimage structures.  To characterize this kind of ``non-invertibility'' and how ``non-invertibility'' contributes to entropy, researchers have introduced various preimage entropy-like quantities from both topological and measure-theoretic perspectives, based on the growth rate of preimages \cite{lp92, hur95, np99, cn05, wz21, w22}.
As for the variational principle for  preimage entropy-like quantities,  Cheng and Newhouse  \cite{cn05}  defined the topological preimage entropy of  $X$ as
$${h}_{pre}(T)=\lim\limits_{\epsilon \to 0}\limsup_{n \to \infty}\limits \frac{1}{n}\sup_{x\in X,{k\geq n}}\limits \log s_n(T, T^{-k}x, \epsilon),$$
where $s_n(T, Z, \epsilon)$  denotes  the maximal cardinality of  $(n,\epsilon)$-separated sets for a non-empty subset $Z$ of $X$. Cheng-Newhouse's variational principle \cite{cn05} states that
$${h}_{pre}(T)=\sup_{\mu \in \mathcal{M}(X,T)}{h}_{pre,\mu}(T),$$ where  ${h}_{pre,\mu}(T)$ is the  measure-theoretic  preimage entropy of $\mu$. 
Unfortunately, the authors in \cite{w23, syz23} showed that there exist certain non-invertible dynamical systems such that
$\sup_{\mu \in \mathcal{M}(X,T)}{h}_{pre,\mu}(T)<{h}_{pre}(T)$. Two other types of pointwise preimage entropy were introduced in \cite{hur95}:
\begin{align*}
{h}_m(T)&=\lim\limits_{\epsilon \to 0}\limsup_{n \to \infty}\limits \frac{1}{n}{\sup_{x\in X}\limits} \log s_n(T, T^{-n}x, \epsilon),\\
{h}_p(T)&={\sup_{x\in X}\limits}\lim\limits_{\epsilon \to 0}\limsup_{n \to \infty}\limits \frac{1}{n} \log s_n(T, T^{-n}x, \epsilon).
\end{align*}  
In 2021, Wu and Zhu \cite{wz21} introduced the pointwise measure-theoretic preimage entropy, and they \cite[Theorem B and Corollary A.1]{wz21} established the following variational principles for pointwise preimage entropy:
$${h}_{p}(T)={h}_{m}(T)=\sup_{\mu \in \mathcal{M}(X,T)}{h}_{m,\mu}(T)<\infty$$
if the non-invertible systems satisfy the uniform separation of preimages. 
For certain non-invertible random dynamical systems, this work was extended to the concept of random pointwise preimage pressure in \cite[Theorem B]{wwz23}. These existing results motivate us to define the proper preimage measure-theoretic entropies such that the classical variational principles hold for the aforementioned three types of preimage entropy-like quantities, without imposing additional assumptions on preimage sets. In particular, we introduce the concept of random preimage metric mean dimension for non-invertible random dynamical systems with infinite entropy. The techniques in \cite{cn05, wz21, wwz23} cannot be adapted to such non-invertible systems for obtaining the   classical variational principle of random preimage metric mean dimension.

\emph{Question 3: For every non-invertible random dynamical system, how to establish the classical variational principle for preimage entropies and preimage metric mean dimensions without imposing the condition of uniform separation of preimages?}

It is the missing counterpart of the measure-theoretic entropy  that renders it difficult to establish proper variational principles for these entropy-like quantities in dynamical systems. Without invoking the underlying dynamics, Bis et al. \cite{bcmp22} established an abstract variational principle for a class of pressure functions on metric spaces in terms of  \emph{finitely additive probability measures}. Inspired by \cite{bcmp22} and Ruelle's idea of topological pressure determining  measure-theoretic entropy  \cite{rue73}, we develop a convex analysis approach to random pressure functions in the context of random dynamical systems, and establish a variational principle for random pressure functions involving \emph{invariant measures}.  This variational principle offers a new perspective on bridging ergodic theory and topological dynamics.

\begin{thm}\label{thm 1.1}
Let $G=\mathbb{Z}_{+}^{k}$, $k\geq 1$, or $G$ be a countable  infinite discrete amenable group, and  let   $\Omega$ be a locally compact, separable metric space  with Borel $\sigma$-algebra $\mathcal{B}(\Omega)$. Let $T$ be an amenable random   (or $\mathbb{Z}_+^k$) dynamical system on a compact metric space $X$  over  an ergodic  measure-preserving system $(\Omega, \mathcal{B}(\Omega),\P, G)$.   If $\Gamma$ is a random pressure  function on $L_{\mathbb{P}}^1(\Omega, C(X))$, then  
$$ \Gamma(0)=\max_{\mu \in \M}s(\mu),$$
where  
 \begin{align*}
s(\mu)&=\inf\{\Gamma(f)-\int fd\mu:f\in L_{\mathbb{P}}^1(\Omega, C(X))\}\\
&= \inf_{f\in \mathcal{A}}\int fd\mu,
\end{align*}
and  $\mathcal{A}=\{ f\in L_{\mathbb{P}}^1(\Omega, C(X)):\Gamma(-f)=0\}$. Furthermore, $s(\mu)$ is a concave and  upper semi-continuous function on the  set $\M$ of  invariant probability measures on $\Omega \times X$. 
\end{thm}

We remark that the  precise definitions of $\Gamma(f)$ and $s(\mu)$ are given in  Section \ref{sec 3}, and the application of Theorem \ref{thm 1.1} to answering the aforementioned three questions is provided in Section \ref{sec 4}. As for the proof of Theorem \ref{thm 1.1}, it differs from the known techniques in \cite{rue73, w82, bcmp22}. For the convenience of readers, we outline the key ideas and tools used in our proof as follows. The fundamental idea for proving Theorem \ref{thm 1.1} is to use the properties of random pressure functions to construct a continuous linear functional, and then link it to an invariant measure with a marginal on the base space. The technical challenges are twofold:

(1) The \emph{Riesz Representation Theorem} does not work in our setting because we need to handle random continuous functions rather than ordinary continuous functions.

(2) The \emph{relativized method} (a modification of Misiurewicz's technique \cite{mis75}), which is used to construct invariant measures for random dynamical systems (as developed in \cite{bog92, kif01}), only applies to measure-theoretic entropy and not to integrals of probability measures. Additionally, unlike $\mathbb{Z}$-actions of dynamical systems, constructing invariant measures with a specific marginal on the base system is critical for random dynamical systems.

These difficulties are overcome by developing a set of novel ideas and introducing some new tools. For (1), we introduce the \emph{Stone vector lattice} to replace the Riesz Representation Theorem. We show that the linear functional that is generated by the \emph{Separation Theorem for Convex Sets} is a \emph{pre-integral} on a space of bounded functions, which ensures the construction of probability measures. For (2), since the base space $\Omega$ may not be compact, we use \emph{outer measure theory} to lift the probability measures obtained from step (1) to Radon measures. Then the convergence of Radon measures allows us to get a probability measure on a Borel $\sigma$-algebra. Finally, we apply \emph{Von Neumann's $L^1$-Ergodic Theorem} to verify that the marginal of this measure on the base space coincides with the given probability measure  $\mathbb{P}$.

The rest of this paper is organized as follows. In section 2,  we review the basic setting  of random dynamical systems and  some  relevant concepts.  In section 3, we introduce the concept of random pressure function and  prove Theorem \ref{thm 1.1}.  In section 4,  we exhibit several applications of  Theorem \ref{thm 1.1}, and answer the Questions 1-3 (corresponding to Theorems A-D).

\section{Preliminary}

In this section, we briefly recall the framework of random dynamical systems and related concepts. A systematic treatment of random dynamical system theory can be found in several monographs \cite{arn98,kif86,cra02,l01,dz15}. 

\subsection{The setup of   amenable random dynamical systems and examples}

Let $G$ be a  group. For a subset $A\subset G$, $|A|$ denotes the cardinality of  $A$. The   symmetric difference of the subsets $A, B \subset G$ is defined as  $A \triangle  B=(A\backslash B)\cup (B\backslash A).$  A group $G$ is called an \emph{amenable group} if there exists  a sequence  $\{F_n\}_{n\geq 1}$ of non-empty finite subsets of $G$ such that  for  every  $g\in G$,
$$\lim_{n\to\infty}\frac{|gF_n\triangle F_n|}{|F_n|}=0.$$

Such a sequence  $\{F_n\}_{n\geq 1}$  is called  a F\o lner sequence of $G$.  Examples of  amenable groups include  all finite groups, Abelian groups, and  all finitely generated groups with sub-exponential growth; free group of rank 2 is non-amenable.    Throughout this  paper,  $G$ is assumed to be either $G=\mathbb{Z}_+^k:=\{(x_1,x_2,...,x_k):x_j \geq 0 ~\forall 1\leq j\leq k\}$(a $k$-dimensional positive integer lattice semigroup with the usual addition)  or   a countable  infinite discrete amenable group.

An  action of $G$ on  a  set $X$  is  a  map  $\alpha: G\times X \rightarrow X$ satisfying
\begin{enumerate}
\item   $\alpha(e_G,x)=x$ for every $x\in X$; 
\item  $\alpha(g_1,\alpha(g_2,x))=\alpha(g_1g_2,x)$ for every $g_1,g_2\in G$ and $x\in X$,
\end{enumerate}
where $e_G=(0,...,0)$ if $G= \mathbb{Z}_+^k$.

For convenience, $\alpha(g,x)$ is often written as $gx:=\alpha(g,x)$, or $\alpha_gx:=\alpha(g,x)$ for every fixed $g$. If  $G$ acts continuously on a compact metric space $X$,  the pair $(X,G)$ is called a (topological) \emph{$G$-system};  if  $G$ acts  measurably on   a probability space $(\Omega, \mathcal{F},\mathbb{P})$ and  $\mathbb{P}(A)=\mathbb{P}(g^{-1}A)$ for all $A\in \mathcal{F}$ and $g\in G$,  the quadruple $(\Omega,\mathcal{F},\mathbb{P}, G)$  is called  a  \emph{measure-preserving $G$-system}. If $G=\mathbb{Z}^k_+$, such systems are non-invertible; for general amenable groups, they are invertible. Additionally, we  remark that a $G$-system $(X,G)$ has  a measurable dynamical model: the  amenability of $G$ guarantees the existence of $G$-invariant Borel probability measures on $X$. Hence, $(X,\mu,\mathcal{B}(X),G)$ is a measure-preserving $G$-system, where $\mathcal{B}(X)$ is  the   Borel $\sigma$-algebra of  $X$.

An \emph{amenable random  dynamical system}(RDS for short) on a compact metric space $(X,d)$  over   a  measure-preserving system $(\Omega,\mathcal{F},\mathbb{P},G)$ is generated by a family of continuous  self-mappings $\{T_{g,\omega}:g\in G,\omega \in \Omega\}$ on  $X$  satisfying
\begin{enumerate}
\item  [(1)] $(\omega,x)\mapsto T_{g,\omega}x$ is measurable   with respect to the product $\sigma$-algebra $\mathcal{F}\otimes \mathcal{B}(X)$\footnote[1]
{The smallest $\sigma$-algebra on $\Omega \times X$ making   the canonical projections $\pi_{\Omega}:\Omega \times X \rightarrow \Omega$ and $\pi_X:\Omega\times X\rightarrow X$ measurable.}  of $\Omega \times X$;
\item [(2)] $T_{e_G,\omega}$ is the identity map on $X$ for all $\omega$;
\item [(3)] $T_{g_1g_2,\omega}=T_{g_2, g_1\omega}\circ T_{g_1,\omega}$ for all $g_1,g_2 \in G$ and $\omega$.
\end{enumerate}

For  $G=\mathbb{Z}_{+}^k$, such a system is called a \emph{random $\mathbb{Z}_{+}^k$-dynamical system}(or a non-invertible random dynamical system). The  measure-preserving system $(\Omega,\mathcal{F},\mathbb{P},G)$   originally models the  ``noise'' considered  in random partial differential equations and stochastic differential equations(cf.\cite{arn98}). Here, it can be viewed as "turbulence" affecting  the $G$-system $(X,G)$ via random choices of continuous transformations on  $X$. This is clarified by defining
$$\alpha:  G\times  \Omega \times X \rightarrow  X, (g,\omega,x)\mapsto \alpha(g,\omega,x):=T_{g,\omega}(x).$$

In particular, if $\Omega$ is a singleton (no ``noise''),  then the amenable random dynamical system reduces to a $G$-system $(X,G)$.

To study dynamics on $X$, we associate   a  $G$-action  on $\Omega \times X$ via  the \emph{skew product transformation}, defined as $\Theta_g(\omega,x)=(g\omega,T_{g,\omega}x)$, which forms  a   $G$-system $(\Omega\times  X,\mathcal{F}\otimes \mathcal{B}(X) ,\Theta, G)$ by  (3).

We present two examples to illustrate random dynamical systems:

\begin{ex}
Let $\mathbb{T}^2$ be the $2$-dimensional torus, and  let $(\Omega, \mathcal{B}(\Omega),\mathbb{P},\theta)$ be an ergodic  Polish system.  For a  measurable mapping $f: \Omega \rightarrow  \mathbb{T}^2$,  the mapping 
$$T_{\omega}x=Ax+f(\omega),~ x\in \mathbb{T}^2$$
generates a random dynamical system on  $\mathbb{T}^2$,
where $ A=\begin{pmatrix}
2 & 1\\
1 & 1
\end{pmatrix}$ is a hyperbolic matrix.
\end{ex}

\begin{ex}
Let $\Omega=\{1,...,k\}^{\mathbb{Z}_{+}}$ be  the full shift over $k$-symbols, with left shift $\sigma$ and  a $\sigma$-invariant measure $\mathbb{P}$ generated by a probability vector $(p_1,p_2,...,p_k)$.  This gives a measure-preserving driving system $(\Omega, \mathcal{B}(\Omega),\mathbb{P},\sigma)$.  Let $\widetilde{\sigma}:X^{\mathbb{N}}\rightarrow X^{\mathbb{N}}$  also be a left shift on the product of   a compact metric space $(X,d)$. For every $\omega=(\omega_0, \omega_1,...)$,  we define the map $T_{\omega}:X^{\mathbb{N}}\rightarrow X^{\mathbb{N}}$ given by $T_{\omega}(x)=\widetilde{\sigma}^{\omega_0}(x)$. Then the iterations
\begin{align*}
T_{\omega}^n=
\begin{cases}
T_{\sigma^{n-1}\omega}\circ T_{\sigma^{n-2}\omega}\circ\cdots \circ T_{\omega},  &\mbox{for}~n\geq 1\\
id,&\mbox{for}~n=0
\end{cases}
\end{align*} 
form a random  $\mathbb{Z}_{+}$-dynamical system.
\end{ex}

We next recall definitions related to invariant measures of random dynamical systems. Let $(X,d)$   be  a compact  metric space. The space of real-valued continuous functions on $X$ is denoted by   $C(X)$, which   becomes a Banach space under the supremum norm  $||\cdot||_{\infty}$.   
For each  real-valued function $f$ on $\Omega \times X$ that is measurable in $(\omega,x)$ and  continuous in $x$ ,  we set 
$$||f||:=\int ||f(\omega)||_{\infty}d\mathbb{P}({\omega}),$$
where $||f(\omega)||_{\infty}=\sup_{x\in X}|f(\omega,x)|$.  The measurability of $||f(\omega)||_{\infty}$  in $\omega$ follows from the separability of $X$ and  the measurability of $f$.  Let $L_{\mathbb{P}}^1(\Omega, C(X))$ denote  the  $L^1$-space of  random continuous functions with $||f||<\infty$.  For instance,
\begin{itemize}
\item [(1)] elements of  $ L^1(\Omega,\mathcal{F},\P)$  define    elements  of $L_{\mathbb{P}}^1(\Omega, C(X))$  via  $(\omega,x)\mapsto f(\omega)$;
\item [(2)] continuous functions $g \in C(X)$ also define  elements of $L_{\mathbb{P}}^1(\Omega, C(X))$ via $(\omega,x)\mapsto g(x)$; 
\item [(3)]  products $(\omega,x)\mapsto f(\omega)g(x)$ also lie in $L_{\mathbb{P}}^1(\Omega, C(X))$.
\end{itemize} 
The space $L_{\mathbb{P}}^1(\Omega, C(X))$ is a Banach space when identifying functions $f$ and $g$ with $||f-g||=0$.

Let $\pi_{\Omega}:\Omega \times X \rightarrow \Omega$ be  the  canonical projection  onto $\Omega$.  A probability measure on $(\Omega \times X, \mathcal{F}\otimes \mathcal{B}(X))$ has marginal $\mathbb{P}$ if  the push-forward  of $\mu$ under $\pi_\Omega$,   denoted by  $(\pi_{\Omega})_{*}\mu$, equals  $\P$. The set of such measures  is denoted by  $\mathcal{P}_{\mathbb{P}}(\Omega \times X)$.  By \cite[Proposition 3.6]{cra02}, any $\mu\in \mathcal{P}_{\mathbb{P}}(\Omega \times X)$ can  disintegrate as  $$d\mu(\omega,x)=d\mu_{\omega}(x)d\P(\omega),$$ where $\mu_{\omega}$ is a  family of conditional probability measures on $X$(unique up to $\P$-a.a $\omega$ equivalence). Endow  the space $\mathcal{P}_{\mathbb{P}}(\Omega \times X)$  with  the weak$^{*}$-topology,  the smallest topology  making  $\mu\mapsto \int fd\mu$ continuous  for all $f\in L_{\mathbb{P}}^1(\Omega, C(X))$. The space $\mathcal{P}_{\mathbb{P}}(\Omega \times X)$ is compact \cite[Lemma 2.1]{kif01}.   The set of invariant measures, denoted by $\mathcal{M}_{\mathbb{P}}(\Omega \times X, G)$, consists of  measures $\mu \in \mathcal{P}_{\mathbb{P}}(\Omega \times X)$  invariant  under the transformation $\Theta_g:\Omega \times X \rightarrow \Omega \times X$ for all $g\in G$.

\section{A convex approach to random pressure function of RDSs}\label{sec 3}

In this section, within the framework of random dynamical systems, we introduce axiomatic definitions for random pressure functions and the corresponding roles with respect to invariant measures. By developing a convex analysis approach to random pressure functions, we establish a variational principle for them (i.e., Theorem \ref{thm 1.1}).

\subsection{An axiomatic definition  for  random pressure functions}
In this subsection, we introduce the  abstract framework of random pressure function in random dynamical systems.

We begin by  reviewing the definition  of Kifer's topological pressure in random dynamical systems \cite{bog92,kif01,dz15}. 
 
For $F \in \FF(G)$ and  $\omega \in \Omega$,  we define  a family of  Bowen metrics  on $X$ defined by 
$$d_{F}^{\omega}(x,y)=\max_{g\in F}d(T_{g,\omega}x, T_{g,\omega}y),$$ 
where $x,y \in X$. The  Bowen open  ball  of  $x$ in $d_F^{\omega}$ with radius $\epsilon$ is  given by 
$$B_F^{\omega}(x,\epsilon)=\{y\in X: d_F^{\omega}(x,y)<\epsilon \}.$$
For $f\in L^1(\Omega, C(X))$,  we put
$$S_Ff(\omega,x)=\sum_{g\in F}f\circ\Theta_g(\omega,x)=\sum_{g\in F}f(g\omega,T_{g,\omega}x).$$
 
A set $E\subset X$ is an  \emph{$(\omega,F,\epsilon)$-separated set} if  for any two distinct $x,y \in E$, $d_{F}^{\omega}(x,y)>\epsilon$.

We define
\begin{align*}
P_F(G,f,d,\omega,\epsilon)=\sup\{\sum_{x\in E}e^{S_Ff(\omega,x)}: E~\mbox{is an }(\omega,F,\epsilon)\mbox{-separated set of }X\}.
\end{align*}
Slightly modifying the proof in \cite[Lemma 1.2]{kif01},   $ P_F(T,f,d,\omega,\epsilon)$  is measurable in $\omega$. Let $\{F_n\}$ be a F\o lner sequence of $G$. It is reasonable to define:
$$P(G,f,d,\{F_n\},\epsilon)=\limsup_{n \to \infty}\frac{1}{|F_n|}\int  \log P_{F_n}(G,f,d,\omega,\epsilon)d\P(\omega).$$

The random topological pressure of $f$  \cite{kif01,dz15} is defined by
$$P(G,f)=\lim_{\epsilon \to 0} P(G,f,d,\{F_n\},\epsilon).$$

\begin{rem}
\item [(1)] It is easy to see that $\lim_{\epsilon \to 0} P(G,f,d,\{F_n\},\epsilon)$ is independent of the choice of the compatible metrics on $X$ and the   F\o lner  sequences of $G$.  Here, it is a bit more involved to verify this independence. So we include a proof for completeness in the Appendix (cf. Proposition \ref{prop 5.3} for details). 
\item [(2)] Unlike the classical topological pressure \cite{w82,op82,st80},  the  random topological pressure $P(G,f)$  depends on the invariant measure $\mathbb{P}$ on $\Omega$.
\item [(3)]  When $f=0$, we let $h_{top}(G,X):=P(G,0)$ and  call $h_{top}(G,X)$ the random topological entropy of $X$.
\item [(4)] In particular,  if $\Omega$ is reduced to a  singleton, then the random topological pressure recovers the  classical topological entropy of amenable groups \cite{op82, st80}.
\end{rem}

The random topological pressure   satisfies the following fundamental properties.

\begin{prop}\label{Prop 3.2}
Let $T=(T_{g, \omega})$ be a RDS 
over the measure-preserving system $(\Omega,\mathcal{F},\mathbb{P},G)$. 
Let $f,g\in L^1(\Omega, C(X))$. Then  
\begin{enumerate}
\item If  $f\leq g$, then  $P(G,f)\leq P(G,g)$.
\item  $ P(G,f+c,d)=P(G,f)+c$ for  any $c\in \R$.
\item  $ h_{top}(G,X)-||f|| \leq P(G,f) \leq h_{top}(G,X)+||f|| $.
\item  $P(G,\cdot): L^1(\Omega, C(X))\longrightarrow \R\cup\{\infty\}$ is either finite value or constantly $\infty$.
\item If $P(G,f)<\infty$ for all  $f\in L^1(\Omega, C(X))$, then 
$$|P(G,f)-P(G,g)|\leq ||f-g||.$$
\item  $P(G,\cdot)$ is convex on $L^1(\Omega, C(X))$. 

\item $P(G,f+g\circ \Theta_h-g)= P(G,f)$ for all $h \in G$.
\end{enumerate}
\end{prop}

\begin{proof}
We prove these statements one by one.

(1) and (2) follow from the definitions of $P(G,\cdot)$. 

(3). Fix $\epsilon >0$. For all  $(\omega ,x) \in \Omega \times X$ and $F\in \FF(G)$, we have
$$e^{-\sum_{g\in F}||f(g\omega)||_{\infty}}\leq e^{S_Ff(\omega,x)}\leq e^{\sum_{g\in F}||f(g\omega)||_{\infty}}.$$

Then, by definitions and the $G$-invariance  of $\mathbb{P}$  we have
\begin{align*}
\int \log s_F(X,d,\omega,\epsilon) d\mathbb{P}(\omega) -|F| \cdot ||f||\leq& \int \log  P_F(G,f,d,\omega,\epsilon) d\mathbb{P}(\omega)\\
\leq& \int \log s_F(X,d,\omega,\epsilon) d\mathbb{P}(\omega) +|F| \cdot ||f||.
\end{align*} 
Dividing both side by $|F_n|$ and taking the $\limsup_{n \to \infty}$,  and then letting $\epsilon \to 0$,
this yields the desired  inequality.

(4). It is a direct consequence of (3): $h_{top}(G,X)<\infty$ if and only if $P(G,f)<\infty$ for all  $f\in L^1(\Omega, C(X))$.

(5).  Fix  $\epsilon >0$, $\omega \in \Omega$ and $F\in \FF(G)$. If $E\subset X$ is an  $(\omega,F,\epsilon)$-separated set of $X$, then we have
$$\sum_{x\in E}e^{S_Ff(\omega, x)}\leq \sum_{x\in E}e^{S_Fg(\omega, x)+\sum_{g \in F}||(f-g)(g\omega)||_{\infty}}.$$
Similar  to (3), we  deduce that  $P(G,f)\leq P(G,g)+||f-g||.$
Exchanging the role of $f$ and $g$, we get  $$P(G,g)\leq P(G,f)+||f-g||.$$
This shows that $P(G,\cdot)$ is  a Lipschitz map on  $L^1(\Omega, C(X))$.

(6). Let $p\in[0,1]$. Let $E$ be an  $(\omega, n,\epsilon)$-separated set of $X$. The  H\"{o}lder's inequality allows us to obtain
\begin{align*}
\sum_{x\in E}e^{pS_Ff(\omega, x)+(1-p)S_Fg(\omega, x)}\leq (\sum_{x\in E}e^{S_Ff(\omega, x)})^p(\sum_{x\in E}e^{S_Fg(\omega, x)})^{(1-p)}.
\end{align*}
Then it follows from this inequality that $P(G,pf+(1-p)g)\leq pP(G,f)+(1-p)P(G,g)$.


(7). Fix  $\epsilon >0$, $\omega \in \Omega$. Let $\{F_n\}$ be a F\o lner sequence of $G$. For every $n\geq 1$ and $x\in X$, by the cocycle property $T_{g_1g_2,\omega}=T_{g_2, g_1\omega}\circ T_{g_1,\omega}$,  we have
$$S_{F_n}(g\circ \Theta_h -g)(\omega, x)=\sum_{s\in hF_n \triangle F_n } g(s\omega, T_{s,\omega}x).$$  If  $E$ is an  $(\omega, n,\epsilon)$-separated set of $X$, then we have
\begin{align*}
e^{-\sum_{s\in hF_n \triangle F_n }||(g(s\omega)||_{\infty}}\cdot\sum_{x\in E}e^{S_{F_n}f(\omega, x)}
\leq&\sum_{x\in E}e^{S_{F_n}(f+g\circ \Theta_h -g)(\omega, x)}\\
\leq &  e^{\sum_{s\in hF_n \triangle F_n }||(g(s\omega)||_{\infty}}\cdot \sum_{x\in E}e^{S_{F_n}f(\omega, x)}.
\end{align*}
This implies that  for every $n\geq 1$,
\begin{align*}
	&\int \log P_{F_n}(G,f,d,\omega,\epsilon) d\mathbb{P}(\omega) -|hF_n \triangle F_n|\cdot ||g||\\
	\leq& \int \log  P_{F_n}(G,f+g\circ \Theta_h-g,d,\omega,\epsilon) d\mathbb{P}(\omega)\\
	\leq& \int \log P_{F_n}(G,f,d,\omega,\epsilon) d\mathbb{P}(\omega) +|hF_n \triangle F_n|\cdot ||g||.
\end{align*} 
Notice that $\lim_{n\to\infty}\frac{|hF_n\triangle F_n|}{|F_n|}=0$, which allows us to deduce the desired equality.





\end{proof}



In  \cite[Definition 2.1]{bcmp22}, the authors introduced an axiomatic definition for pressure functions.  More precisely, let $(X,d)$ be a locally compact  metric space and  $\textbf{B}(X)$ be a Banach space  over $\mathbb{R}$. A map $\gamma: \textbf{B}(X) \rightarrow \mathbb{R} $ is a \emph{pressure function} if it satisfies  the following condition:
\begin{enumerate}
\item monotonicity:  $f\leq g$ $\Longrightarrow$ $\gamma(f)\leq \gamma(g)$ $ \forall f,g \in  \textbf{B}(X)$.
\item  Translation invariance:  $\gamma(f+c)= \gamma(f)+c$ $\forall f\in  \textbf{B}(X)$ and $c\in \mathbb{R}$;
\item Convexity:  $\gamma(pf+(1-p)g)\leq p\gamma(f)+(1-p)\gamma(g)$  $\forall f, g\in \textbf{B}(X)$ and $p\in [0,1]$.
\end{enumerate}
Based on the properties of random topological pressure established in Proposition \ref{Prop 3.2}, we formulate an axiomatic definition of \emph{random pressure functions} on $L_{\mathbb{P}}^1(\Omega, C(X))$.  Such a treatment  helps us emphasize the aforementioned questions using a unified approach. 

\begin{df}\label{Def3.1}
Let $T=(T_{g, \omega})$ be a RDS 
over the measure-preserving system $(\Omega,\mathcal{F},\mathbb{P},G)$. 
A  function $\Gamma: L_{\mathbb{P}}^1(\Omega, C(X))\rightarrow \mathbb{R}$  is  called a  \emph{random pressure function} if it satisfies the following conditions:
\begin{enumerate}
\item  Monotonicity:  $f\leq g \Longrightarrow \Gamma(f)\leq \Gamma(g)$ $ \forall f,g \in L_{\mathbb{P}}^1(\Omega, C(X))$;
\item  Translation invariance: $\Gamma(f+c)= \Gamma(f)+c $ $ \forall f \in L_{\mathbb{P}}^1(\Omega, C(X))$ and $ c\in \mathbb{R}$;
\item Convexity:    $\Gamma(pf+(1-p)g)\leq p\Gamma(f)+(1-p)\Gamma(g)$  $\forall f, g\in L_{\mathbb{P}}^1(\Omega, C(X))$ and $p\in [0,1]$;
\item Lipschitz: 
$|\Gamma(f)-\Gamma(g)|\leq ||f-g||$ $ \forall f,g \in L_{\mathbb{P}}^1(\Omega, C(X))$;
\item Semi-cohomology:  $\Gamma(f\circ\Theta_g)\leq \Gamma(f)$ $ \forall f \in L_{\mathbb{P}}^1(\Omega, C(X))$ and $g \in G$.
\end{enumerate}

\end{df}
\begin{rem}
\begin{enumerate}
\item [(i)] A random pressure function  $\Gamma$ satisfies the \emph{cohomology property} if  the equality  holds in (5).
\item [(ii)] Only the semi-cohomology property involves the dynamics; the remaining properties are independent of the choice of the specific random dynamical system.   Semi-cohomology property  plays a vital role in our proof  for establishing the variational principle for random pressure functions in terms of invariant measures.
\end{enumerate}
\end{rem}

\subsection{An axiomatic definition  for  the measure-theoretic quantity $s(\mu)$} 

Next, we explain how to define a proper measure-theoretic quantity for invariant measures using a random pressure function.

For an expansive dynamical system $(X,T)$ with  the specification property, Ruelle \cite[Theorem 5.1, p.245]{rue73} first established  a   variational principle for classical topological pressure:
$$P(T,f)=\max_{\mu \in \mathcal{M}(X,T) }\{s(\mu)+\int fd\mu\},$$
where  $s(\mu)=\inf_{f\in C(X)}\{P(T,f)-\int fd\mu\}$.

For any topological dynamical systems $(X,T)$ with finite topological entropy,  Walters \cite[Theorem  9.12]{w82} verified that  for any  $\mu_0 \in \mathcal{M}(X,T)$,   the entropy map $\mu \in \mathcal{M}(X,T)~\mapsto h_{\mu}(T)$  is  upper semi-continuous (u.s.c.) at   $\mu_0$ if and only if
$h_{\mu_0}(T)=s(\mu_0).$
This suggests that the measure-theoretic entropy of $\mu_0$ is completely determined by topological pressure in this case. Now, we inject the above ideas into the random dynamical systems, and introduce a measure-theoretic quantity for invariant measures.



\begin{df}\label{def 2.2}
Let $T=(T_{g, \omega})$ be a RDS 
over the measure-preserving system $(\Omega,\mathcal{F},\mathbb{P},G)$, and let $\Gamma$ be a random pressure function on $L_{\mathbb{P}}^1(\Omega, C(X))$.  For every  $\mu \in \mathcal{P}_{\mathbb{P}}(\Omega \times X)$,     the measure-theoretic  entropy  of  $\mu$  with respect to $\Gamma$  is defined as 
$$s(\mu):=\inf\{\Gamma(f)-\int fd\mu:f\in L_{\mathbb{P}}^1(\Omega, C(X))\}.$$
\end{df}

We now present equivalent characterizations of  $s(\mu)$ using specific convex subdomains of $L_{\mathbb{P}}^1(\Omega, C(X))$.

\begin{prop} \label{prop 2.4}
Under the setting of  Definition \ref{def 2.2},   for every  $\mu \in \mathcal{P}_{\mathbb{P}}(\Omega \times X)$,  we have
\begin{align*}
s(\mu)&=\inf_{f\in \mathcal{A}}\int fd\mu\\
&= \inf_{f\in \widetilde{\mathcal{A}}}\int fd\mu,
\end{align*}
where  $$\mathcal{A}=\{f\in L_{\mathbb{P}}^1(\Omega, C(X)):\Gamma(-f)=0\}$$
and 
$$\widetilde{\mathcal{A}}=\{f\in L_{\mathbb{P}}^1(\Omega, C(X)):\Gamma(-f)\leq0\}.$$
\end{prop}

\begin{proof}
We  first show  $ s(\mu)=\inf_{f\in \mathcal{A}}\int fd\mu$. This essentially follows from the  translation invariance    of $\Gamma$.  Indeed,   if $f\in \mathcal{A}$, one has
\begin{align*}
\int fd\mu=\Gamma(-f)-\int -fd\mu\geq s(\mu).
\end{align*}
Thus,  $ s(\mu)\leq\inf_{f\in \mathcal{A}}\int fd\mu$. 
If $f\in L_{\mathbb{P}}^1(\Omega, C(X))$, then $\Gamma(-(\Gamma(f)-f))=0$. This shows $ \Gamma(f)-f \in \mathcal{A}$ and hence
$$\Gamma(f)-\int fd\mu \geq  \inf_{f\in \mathcal{A}}\int fd\mu.$$
We get $ s(\mu)\geq\inf_{f\in \mathcal{A}}\int fd\mu$.

The first equality $ s(\mu)=\inf_{f\in \mathcal{A}}\int fd\mu$  implies that $s(\mu)\geq \inf_{f\in  \widetilde{\mathcal{A}}}\int fd\mu$.  Since $L_{\mathbb{P}}^1(\Omega, C(X))$ is a linear space, $-f \in L_{\mathbb{P}}^1(\Omega, C(X))$ whenever $f\in L_{\mathbb{P}}^1(\Omega, C(X))$.   This shows that
\begin{align*}
s(\mu )&=\inf\{\Gamma(f)-\int fd\mu:f\in L_{\mathbb{P}}^1(\Omega, C(X))\}\\
&=\inf\{\Gamma(-f)+\int f d\mu:f\in L_{\mathbb{P}}^1(\Omega, C(X))\}\\
&\leq \inf_{f\in  \widetilde{\mathcal{A}}}\int fd\mu.
\end{align*}
\end{proof}

Let $(\Omega,d)$ be a metric space.   A  function $f:\Omega\rightarrow \mathbb{R}$ is called $\emph{compactly supported}$  if its support 
$${\rm spt}(f):=\overline{\{\omega\in \Omega: f(\omega)\not=0\}}$$
is a compact subset of $\Omega$. The space of  compactly supported  continuous functions on $\Omega$ is denoted by $C_c(\Omega)$.   Note that  a continuous function  $f\in C_c(\Omega)$ if and only if there exists a compact subset $K\subset \Omega$ such that $f|_{\Omega \backslash K}=0$. Using this fact, $C_c(\Omega)$ is a real linear space.  By Urysohn's lemma,  $C_c(\Omega)$ is non-empty for  locally compact metric space $\Omega$. 

\begin{prop}\label{prop 2.10}
 Let $(\Omega,\mathcal{F},\mathbb{P})$ be a probability  space. The following subsets are dense in $L_{\mathbb{P}}^1(\Omega, C(X))$.

(1) The set of finite linear combinations of products of characteristic functions of measurable subsets of $\Omega$ and  continuous functions on $X$  is  dense in $L_{\mathbb{P}}^1(\Omega, C(X))$.

(2) The set of bounded functions of  $L_{\mathbb{P}}^1(\Omega, C(X))$ is  dense in $L_{\mathbb{P}}^1(\Omega, C(X))$.

(3) If  $\Omega$ is a locally compact, separable  metric space with  Borel $\sigma$-algebra $\mathcal{B}(\Omega)$, then $C_c(\Omega\times X)$ is dense in $L_{\mathbb{P}}^1(\Omega, C(X))$.
\end{prop}

\begin{proof}
(1). For $f\in L_{\mathbb{P}}^1(\Omega, C(X))$, we  define a map $$\Hat{f}: (\Omega ,\mathcal{F})\rightarrow  (C(X), \mathcal{B}(C(X)))$$ given by  $\Hat{f}(w)=f(\omega,\cdot)\in C(X)$,  where  $\mathcal{B}(C(X))$  is  the  Borel $\sigma$-algebra of $C(X)$ generated by all open balls of $C(X)$ formed by $$B(h,r)=\{g\in C(X): ||g-h||_{\infty}<r\}$$ with $h\in C(X), r>0$. Then $\Hat{f}$ is measurable. Indeed, let $E$ be  a  countable dense  subset  of $X$. Then  for all $\omega$,
\begin{align*}
||f(\omega, \cdot)-h||_{\infty}=\sup_{x\in X}|f(\omega, x)-h(x)|
=\sup_{y\in E}|f(\omega, y)-h(y)|.
\end{align*}
Hence 
\begin{align*}
\Omega\backslash \Hat{f}^{-1}(\overline{B}(h,r))
=&\{\omega \in \Omega: ||f(\omega, \cdot)-h||_{\infty}> r\}\\
=&\mathop{\bigcup}_{y\in E}\{\omega \in \Omega: |f(\omega, y)-h(y)|> r\}\in \mathcal{F},
\end{align*}
which implies that $\Hat{f}$ is  measurable.

Let  $\{g_n\}_n$ be a  countable dense subset of $C(X).$ For $n,k \geq 1$,  we define 
\begin{align*}
A_{n,k}&=\{g\in C(X):||g-g_n||_{\infty}<\frac{1}{k}\},\\
B_{1,k}&=A_{1,k}, ~B_{n,k}= A_{n,k}\backslash \bigcup_{j=1}^{n-1}A_{j,k}, ~ n\geq 2.
\end{align*}
Then  $\{\Hat{f}^{-1}(B_{n,k}):n\geq1\}$ is  a  measurable partition of $\Omega$.  For each $k\geq 1$, there exist  a  positive  integer  $n_k $  and  a $\P$-measurable set $\Omega_{n_k}=\Omega\backslash\bigcup_{j=1}^{n_k}\Hat{f}^{-1}(B_{j,k})$ such that  $\P(\Omega_{n_k})<\frac{1}{k}$. We set
\begin{align*}
h_k(\omega,x)
=\sum_{j=1}^{n_k}\chi_{\Hat{f}^{-1}(B_{j,k})}(\omega)g_j(x)
+\chi_{\Omega_{n_k}}\cdot 0.
\end{align*}
Therefore, 
\begin{align*}
||f-h_k||&=\int_{\Omega\backslash\Omega_{n_k}}||f(\omega)-h_k(\omega)||_{\infty}d\P(\omega)
+ \int_{\Omega_{n_k}}||f(\omega)||_{\infty}d\P(\omega)\\
&\leq\int_{\Omega_{n_k}}||f(\omega)||_{\infty}d\P(\omega)+ \frac{1}{k} . 
\end{align*}
Since $||f||<\infty$, by the absolute continuity of integrals one has   $\int_{\Omega_{n_k}}||f(\omega)||_{\infty}d\P(\omega) \to 0$  as $k$  goes to $\infty$.   This shows  (1).

(2) is a direct consequence of (1).

(3)  Notice that  the product $\sigma$-algebra  $\mathcal{B}(\Omega)\otimes \mathcal{B}(X)$ coincides with the Borel $\sigma$-algebra  $\mathcal{B}(\Omega \times X)$ since  $\Omega \times X$ has  a  countable topological basis. Clearly, $C_c(\Omega\times X)\subset L_{\mathbb{P}}^1(\Omega, C(X))$. By (1), it suffices to approximate functions of the form  $\chi_A\cdot g$, $A\in \mathcal{B}(\Omega)$ with elements of $C_c(\Omega\times X)$.

Since   $\Omega$ is a locally compact and separable  metric space, it is $\sigma$-compact.  By \cite[Theorem B.13, p.425]{ew13}, the probability measure $\P$ on $\Omega$ is regular, i.e., 
$$\P(A)=\sup_{K\subset A, ~K \text{compact~ set}}\P(K)= \inf_{A\subset U,~ U \text{open~set}}\P(U).$$  Fix  $\epsilon>0$ and  choose a  compact set $K$ and an open set $U$ with $K\subset A\subset U$  and $\P(U\backslash K)<\frac{\epsilon}{2M}$, where $M=\max_{x\in X}|g(x)|+1$. By  Urysohn's lemma for locally  compact space \cite[Theorem 6.19]{rud87}, there exists $f\in C_c(\Omega )$ such that $0\leq f(\omega)\leq1$, $f|_K\equiv1$ and  the compact support $\text{spt}(f)\subset U$.  Since $f(\omega)g(x)=0$ for any  $(\omega,x) \in \Omega \times X \backslash \text{spt}(f)\times \text{spt}(g)$, then $f\cdot g\in  C_c(\Omega\times X)$  and 
\begin{align*}
||\chi_Ag-fg||=\int_{\Omega}\sup_{x\in X}|\chi_A(\omega)g(x)-f(\omega)g(x)|\P(\omega)<\epsilon.
\end{align*}
This shows  that  $C_c(\Omega\times X)$ is  a dense subset of $L_{\mathbb{P}}^1(\Omega, C(X))$. 
\end{proof}

Compared with Proposition \ref{prop 2.4}, using Proposition \ref{prop 2.10} we can equivalently express  $s(\mu)$ by shrinking the function space  $L_{\mathbb{P}}^1(\Omega, C(X))$ to the  compactly supported function space $C_c(\Omega\times X)$ and  expanding the  range  of  $\Gamma(-\cdot)$ to non-positive values.

\begin{lem}\label{lem 2.6}
Suppose that  $\Omega$ is a locally compact and separable  metric space.
Let $T=(T_{g,\omega})$ be a RDS over the measure-preserving system $(\Omega,\mathcal{B}(\Omega),\mathbb{P},G)$, and let $\Gamma$ be a random pressure function on $L_{\mathbb{P}}^1(\Omega, C(X))$.  Then   for any  $\mu \in \mathcal{P}_{\mathbb{P}}(\Omega \times X)$, 
$$s(\mu)=\inf_{f\in \widetilde{\mathcal{A}_c}}\int fd\mu,$$ where $ \widetilde{\mathcal{A}}_c=\{f\in C_c(\Omega\times X):\Gamma(-f)\leq0\}$.
\end{lem}

\begin{proof}
Fix   $\mu \in \mathcal{P}_{\mathbb{P}}(\Omega \times X)$. It  follows from Proposition \ref{prop 2.4} that we  only need to show 
$$s(\mu)= \inf_{f\in \widetilde{\mathcal{A}}}\int fd\mu
\geq \inf_{f\in \widetilde{\mathcal{A}_c}}\int fd\mu.$$

 The map $f \in L_{\mathbb{P}}^1(\Omega, C(X))\mapsto \int fd\mu$  is continuous. Indeed,  for any $f_1,f_2\in   L_{\mathbb{P}}^1(\Omega, C(X))$,   one has
\begin{align*}
|\int f_1d\mu-\int f_2d\mu|&\leq \int_{\Omega}\int_{X}|f_1(\omega,x)-f_2(\omega,x)|d\mu_{\omega}(x)d\P(\omega)\\
&\leq \int||(f_1-f_2)(\omega)||_{\infty}d\P(\omega)\\
&=||f_1-f_2||.
\end{align*}
Let  $f\in \widetilde{\mathcal{A}}$ and $\gamma >0$. Proposition \ref{prop 2.10},(3)  allows us to choose  a  sequence   $\{f_n\}$ of  compactly supported continuous functions of $L_{\mathbb{P}}^1(\Omega, C(X))$  such that $f_n \to f+\frac{\gamma}{2}$.  Then 
$$\int f_nd\mu  \to \int f+\frac{\gamma}{2}d\mu,$$ and by the Lipschitz and translation invariance properties of  $\Gamma$, this yields that
$$\Gamma(-f_n) \to \Gamma(-(f+\frac{\gamma}{2}))<0.$$
Thus,  for sufficiently large  $N$ one has $f_{N}\in \widetilde{\mathcal{A}}_c$  and  $\int f_{N}d\mu  < \int f d\mu+\gamma$. So  
$\inf_{f\in\widetilde{\mathcal{A}_c}}\int fd\mu\leq  \int f d\mu$, which implies the desired inequality. 
\end{proof}

\subsection{Variational principle of random pressure function}
In this subsection,  we introduce several powerful tools and prove Theorem \ref{thm 1.1}.

Notice that $s(\mu)$ differs significantly from (fiber) measure-theoretic entropy. Consequently, the methods used in previous works \cite{rue73,w82,kif01,bcmp22} cannot be directly applied here, as we are now dealing with integrals of invariant measures rather than measure-theoretic entropy. Additionally, in the random setting, we need to construct invariant measures with marginal $\mathbb{P}$, which is not involved in the deterministic systems \cite{rue73,bcmp22}.  The difficult part  for us is to get the inequality:
$$ \Gamma(0)\leq\max_{\mu \in \M}s(\mu).$$ 
The central problem  is  how to  construct a  continuous linear  functional using  the properties of the random pressure function and relate it to a probability measure on $\Omega \times X$ with marginal $\mathbb{P}$.

This can be achieved through the following steps:\\
(a) Utilize the separation theorem for convex sets to obtain a linear functional.\\
(b) Introduce the \emph{Stone vector lattice} to handle random continuous functions, and use the \emph{pre-integral} in place of the Riesz representation theorem.\\
(c) Employ \emph{outer measure theory} to resolve issues arising from the convergence of $M(\Omega \times X)$, and   obtain a probability measure with   marginal  $\mathbb{P}$ on $\Omega$ satisfying the desired inequality.

Before proceeding, we first invoke the three powerful tools mentioned above.

\subsubsection{Some auxiliary  tools}

The following \emph{separation theorem for convex sets}, presented in \cite[p.417]{ds88}, enables the separation of two disjoint closed convex sets by a suitable linear functional.

\begin{lem}\label{lem 3.4}
Let $V$ be  a locally convex  real linear Hausdorff
topological space. Suppose that $K_1, K_2$ are disjoint closed   convex subsets, and  $K_1$ is compact. Then  there exists a continuous real-valued  linear  functional $L$ on $V$ such that $L(x)<L(y)$ for any $x\in K_1$ and $y\in K_2$. 
\end{lem}

By the assumption on $\Omega$ in  Theorem \ref{thm 1.1}, $\Omega\times X$ is a locally compact metric space. Although the Riesz representation theorem can be applied to construct a Radon measure on $\Omega \times X$  for which  weak$^{*}$ convergence  holds for  all  $f\in C_c(\Omega \times X)$,  the composition of    compactly supported continuous functions  on $\Omega \times X$   with the iterations of the skew product transformation 
$\Theta_g$  may not lie in $C_c(\Omega \times X)$. This is because  the skew product transformation $\Theta_g$ is merely measurable.  To overcome this difficulty, we introduce the \emph{pre-integral} to replace the role of the Riesz representation theorem.

Let $\mathcal{L}$ denote a family of real-valued functions on a set $X$.  $\mathcal{L}$ is called a \emph{vector lattice} if  it  is   a  linear space, and  $f\vee g:=\max\{f,g\}\in \mathcal{L}$ for all $f,g \in  \mathcal{L}$. Additionally,  $\mathcal{L}$  is said to be a \emph{Stone vector lattice} if  it is a vector lattice  and   $f\wedge 1:=\min\{f,1\}\in \mathcal{L}$ for all $f \in  \mathcal{L}$.  

Given a set $X$ and a vector lattice $\mathcal{L}$ on $X$, a function $L:\mathcal{L} \rightarrow \mathbb{R}$ is called a  \emph{pre-integral} if it is linear, positive, and satisfies $L(f_n)\downarrow 0$ whenever $f_n\in \mathcal{L}$ and $f_n(x)\downarrow 0$ for all $x\in X$. 

The following lemma, from \cite[Theorem 4.5.2]{bud02}, serves as an analogue of the Riesz representation theorem in a more general framework:
\begin{lem}\label{lem 3.5}
Let $X$ be a set, and  let $L$ be a pre-integral  on a  Stone vector lattice  $\mathcal{L}$. Then  there exists a measure $\mu$ on $(X,\sigma(\mathcal{L}))$  such that for all  $f\in \mathcal{L}$,
$$L(f)=\int fd\mu,$$
where $\sigma(\mathcal{L})$ is  the smallest  $\sigma$-algebra  on $X$ such that all functions in  $\mathcal{L}$ are measurable.
\end{lem}

Using pre-integrals, one can construct a measure on the locally compact metric space $\Omega \times X$. To obtain an invariant measure, a standard technique involves iterating the measure under $\Theta$. However, the usual weak$^{*}$ topology convergence for Borel probability measures on compact sets fails in this setting. Instead, we rely on a type of convergence for locally compact metric spaces from outer measure theory to overcome this obstruction.

We proceed to present some useful concepts and results from outer measure theory \cite{fol99,m95,sim18}. 
Given a set $X$, a set function $\mu: 2^X\rightarrow[0,\infty]$ is  called an \emph{outer measure}   on  $X$  if  it satisfies $\mu(\emptyset)=0$,   $\mu(A)\leq \mu(B)$ for all  $A\subset B$ and $\mu(\bigcup_{j=1}^{\infty}A_j)\leq \sum_{j=1}^{\infty}\mu(A_j)$  for any sequence $\{A_j\}_{j\geq 1} \subset 2^X$.

Examples include the standard Lebesgue measure on $\mathbb{R}^n$, and the counting measure on  $X$ given by $\mu(A)=\#A$ for any  finite  subset $A$ of $X$ and $\mu(A)=\infty$ otherwise.

If  a subset $A  \subset X$ satisfies Carath\'{e}odory's  condition: for any $S  \subset X$, $$\mu(A)=\mu(A\cap S)+\mu(A\backslash S),$$ 
then  $A$ is called a \emph{$\mu$-measurable set}.  In particular, an outer measure $\mu$ on   a metric space $X$ is called  \emph{Borel-regular} if   all  Borel sets of $X$ are  $\mu$-measurable sets,  and  for any $A\subset X$, there exists  a Borel  set  $B\supset A$  such  that $\mu(B)=\mu(A)$.

For a   locally  compact metric space $X$, an  outer measure  $\mu$  is called  a \emph{Radon measure} if it satisfies:
\begin{enumerate}
\item [(1)] $\mu$ is Borel-regular;
\item  [(2)]$\mu(K)<\infty$  for each compact set $K\subset X$;
\item [(3)]for any  $A\subset X$,  $\mu(A)=\inf_{U\text{open~ set}, A\subseteq U}\limits \mu(U)$;
\item [(4)]  for any  open set $U\subset X$, $\mu(U)=\sup_{K\text{compact~set}, K\subset U}\limits \mu(K)$.
\end{enumerate}

A sequence  of Radon measures $\{\mu_n\}_{n= 1}^\infty$ on $X$  \emph{converges  weakly  to $\mu$}  if $\lim\limits_{n\to \infty}\int f d\mu_n=\int fd\mu$ for all $f \in C_c(X)$.   The following lemma is a direct consequence of Urysohn's lemma  for locally compact  metric spaces.

\begin{lem}\cite[Theorem 1.24]{m95}\label{lem 3.1}
Let  $X$ be a locally compact metric  space. Suppose that 
$\{\mu_n\}_{n=1}^{\infty}$  and  $\mu$ are  Radon measures on    $X$ such that  $\mu_n $ weakly converges  to $\mu$. Then 
\begin{enumerate}
\item[(i)] for any  open set $O \subset X$, $\liminf_{n\to \infty}\limits \mu_n(O)\geq \mu(O)$;
\item [(ii)] for any  compact subset $K\subset X$, $\limsup_{n\to \infty}\limits \mu_n(K)\leq \mu(K)$.
\end{enumerate}

\end{lem}

We give  a criterion for determining whether a sequence of Radon measures admits a weakly convergent subsequence.

\begin{lem}\cite[Theorem 5.15]{sim18}\label{lem 3.2}
Let  $X$ be a locally compact and $\sigma$-compact metric   space. Let $\{\mu_n\}_{n=1}^{\infty}$  be a sequence of Radon measures on  $X$ such that  $\sup_{n\geq 1}\mu_n(K)<\infty$ for each compact set $K \subset X$.  Then there exists a subsequence $\{\mu_{n_k}\}_{k\geq 1}$ that converges weakly to a Radon measure $\mu$ on $X$.
\end{lem}

Carath\'{e}odory's  theorem \cite[Proposition  1.11]{fol99} states that the collection of all  $\mu$-measurable sets in $X$, denoted by $\sigma(\mu)$, forms a $\sigma$-algebra. Restricting a Radon measure (with $\mu(X)=1$) to  its Borel $\sigma$-algebra  yields a Borel probability measure. Conversely,   given a Borel probability measure, one can construct a Radon measure as follows:

\begin{lem}\label{lem 3.3}
Let $X$ be a locally compact, separable metric space, and  let $\mu$ be a Borel probability measure on  $X$.  For any $A\subset X$,  we define
$$\mu^{*}(A)=\inf_{B\in \mathcal{B}(X),~ A\subset B}\mu(B).$$
Then $\mu^{*}$ is a Radon measure on $X$. If $f$ is a  $\mu$-integrable  function on $X$, then  
\begin{align}\label{equ 3.1}
\int fd\mu=\int fd\mu^{*}.
\end{align}
\end{lem}

\begin{proof}
It is easy to see that  $\mu^{*}$ is an outer measure on $X$.  To show  that   $\mu^{*}$ is a Radon measure on $X$, by \cite[Lemma 5.9]{sim18} it suffices to show  that \\
(a) each open  set  of $X$ is a  countable union of compact sets of $X$;\\
(b) $\mu^{*}$ takes finite value  for  compact sets of $X$;\\
(c) $\mu^{*}$ is  a Borel-regular outer measure on $X$.

(b) is clear since $\mu(X)=1$. Since  each open set of $X$ is a  $F_\sigma$-set and  any closed set of $\sigma$-compact  metric space  $X$ is still  $\sigma$-compact,  then each open  set  can be expressed as  a  countable union of some  compact sets of $X$.  This shows (a). For any $A\subset X$ and $n\geq 1$,  one can choose  a Borel set  $B_n \supset A$ such that
$\mu(B_n)<\mu^{*}(A)+\frac{1}{n}$.  We set $B=\bigcap_{n\geq 1}B_n \in \mathcal{B}(X)$. Then $$\mu^{*}(A)\leq \mu^{*}(B)\leq \mu(B)  < \mu^{*}(A)+\frac{1}{n},$$
which  implies that $\mu^{*}(A)=\mu^{*}(B)$. Now assume that $A\in \mathcal{B}(X)$ and $S\subset X$. By the countable sub-additivity of outer measure,  we have $\mu^{*}(S)\leq \mu^{*}(S\cap A)+\mu^{*} (S\backslash A)$. On the other hand, there exists  a  Borel set $ S_1 \supset S$ such that $\mu^{*}(S)=\mu^{*}(S_1)$. Noting that $\mu^{*}|_{\mathcal{B}(X)}=\mu$, this yields that 
\begin{align*}
\mu^{*}(S)=\mu(S_1)&= \mu(S_1\cap A)+\mu (S_1\backslash A)\\
&=\mu^{*}(S_1\cap A)+\mu^{*} (S_1\backslash A)\\
&\geq  \mu^{*}(S\cap A)+\mu^{*} (S\backslash A).
\end{align*}
So $\mu^{*}$ is Borel-regular. This shows (c). To sum up, $\mu^{*}$ is a Radon measure on $X$.

Recall that the integral   $\int fd\mu^{*}$ of $f$  with respect to $\mu^{*}$   is defined on  $(X, \sigma(\mu^*),\mu^{*}|_{\sigma(\mu^{*})})$ as in  classical measure theory. Each  $\mu$-measurable  function  $f$ on $X$  is also   $\mu^*$-measurable since  $ \sigma(\mu^*) \supset  \mathcal{B}(X)$.  Then the  equality (\ref{equ 3.1}) follows from  the  fact $\mu^{*}|_{\mathcal{B}(X)}=\mu$ and  a standard approximation  technique for integrable functions that is
 used  in real analysis.
\end{proof}

\subsubsection{Proof of Theorem \ref{thm 1.1}}
\begin{proof} 

The translation invariance    of    $\Gamma$ shows that  $\Gamma(-\Gamma(0))=\Gamma(0)-\Gamma(0)=0$.  
By Proposition \ref{prop 2.4}, for all $\mu\in\PP$ we have $s(\mu)\leq \int \Gamma(0)d\mu=\Gamma(0)$. This implies that 
$$\sup
_{\mu \in \mathcal{M}_{\mathbb{P}}(\Omega \times X, G)}s(\mu)\leq \Gamma(0).$$

To get the converse inequality
\begin{align}\label{equ 3.4}
\Gamma(0)\leq\sup_{\mu \in \mathcal{M}_{\mathbb{P}}(\Omega \times X, G)}s(\mu),
\end{align}
we need  to show  for any $\epsilon >0$, there exists $\mu\in \mathcal{M}_{\mathbb{P}}(\Omega \times X, G)$ such that  $\Gamma(0) -\epsilon<s(\mu)$.  By Lemma \ref{lem 2.6}, this is equivalent to show 
\begin{align}\label{equ 3.5}
\inf_{f\in \widetilde{\mathcal{A}}_c}\int fd\mu+\int c d\mu+\epsilon >0,
\end{align}
where 
$ \widetilde{\mathcal{A}}_c=\{f\in C_c(\Omega\times X):\Gamma(-f)\leq0\}$
and $c:=-\Gamma(0)$.  Equip  the bounded function space 
$$L_{b}^1(\Omega, C(X))=\{f\in L_{\mathbb{P}}^1(\Omega, C(X)):f~\text{is bounded on } \Omega \times X\}$$  
with subspace topology inherited  from $L^{1}(\Omega, C(X))$, and  define 
$$ \widetilde{\mathcal{A}}_b=\{f\in L_{b}^1(\Omega, C(X)):\Gamma(-f)\leq0\}.$$
Clearly,  $ \widetilde{\mathcal{A}}_c\subset \widetilde{\mathcal{A}}_b$.
The convexity  and Lipschitz properties of $\Gamma$     yield that $\widetilde{\mathcal{A}}_b$ is a closed convex subset of $L_b^1(\Omega, C(X))$.   Notice that  $\Gamma(c)=0$. Then $-(c+\frac{\epsilon}{2})\notin \widetilde{\mathcal{A}}_b$ and hence $-c\notin \widetilde{\mathcal{A}}_b+\frac{\epsilon}{2}$.  Consider  the  disjoint  closed  convex subsets $K_1=\{-c\}$ and $K_2=\widetilde{\mathcal{A}}_b+\frac{\epsilon}{2}$.  By Lemma \ref{lem 3.4},  there exists a  continuous real-valued linear  functional  $L$ on $L_b^1(\Omega, C(X))$ such that   for any  $f\in \widetilde{\mathcal{A}}_b+\frac{\epsilon}{2}$,
\begin{align}\label{equ 3.7}
L(f)> L(-c).
\end{align}

We   first  obtain a Borel  probability  measure $\mu_0$ on $(\Omega\times X,\mathcal{B}(\Omega \times X) )$ such that for  every $f \in  L_b^1(\Omega, C(X))$, 
\begin{align}\label{equu 3.5}
\frac{L(f)}{L(1)}=\int fd\mu_0.
\end{align}  
Let $f\geq 0$ and $\lambda >0$. By the translation invariance and the  monotonicity of $\Gamma$, one has
\begin{align*}
\Gamma(-(\lambda f+1+\Gamma(0)))&=\Gamma(-\lambda f)-1-\Gamma(0)\\
&\leq \Gamma(0)-1-\Gamma(0)=-1<0.
\end{align*}
Then $L(-c)<\lambda L(f)+L(1+\Gamma(0)+\frac{\epsilon}{2})$ by (\ref{equ 3.7})  and the linearity of $L$.    If $L(f)<0$,  by letting $\lambda\to +\infty$  we get $L(-c)=-\infty$. This implies that $L(f)$ must be non-negative.   So  $L$ is  a  positive linear functional.  Let $\{g_n(\omega,x)\}_n$ be a  sequence  of bounded functions in $L_b^1(\Omega, C(X))$  such that   $g_n(\omega,x)$ is pointwise  decreasing  to  $0$. For every fixed $\omega$, the  function sequence  $\{g_n(\omega)\}_n$ on $X$, given by  $g_n(\omega):=g_n(\omega,\cdot)$,   is pointwise decreasing to  $0$ for all $x\in X$.  By  Dini's theorem,  $g_n(\omega) $ converges  uniformly to $0$ for every $\omega$, i.e., $||g_n(\omega)||_{\infty}\downarrow 0$.  Then  Levi's  monotone  convergence theorem  gives us  $$\lim_{n\to\infty}||g_n-0||=\lim_{n\to\infty}\int ||g_n(\omega)||_{\infty}d\P(\omega)=0.$$ Since $L$ is positive, continuous and  $L(0)=0$, this yields that $L(g_n)\downarrow 0$.  Since  $L$  is positive and  take non-zero value for some $f$,   we can choose $ g\in L_b^1(\Omega, C(X))$ such that $0\leq g\leq 1$ and  $L(g)>0$.  Then  $$L(1)=L(g)+L(1-g)>0.$$

To sum up, $\frac{L(\cdot)}{L(1)}$ is  a  pre-integral on   Stone vector lattice $L_b^1(\Omega, C(X))$.  By Lemma \ref{lem 3.5},  there exists a  probability measure $\mu_0$ on $(\Omega\times X,\sigma(L_{b}^1(\Omega, C(X)) )$ such   that for  every $f \in  L_b^1(\Omega, C(X))$, 
\begin{align}\label{equ 3.8}
\frac{L(f)}{L(1)}=\int fd\mu_0.
\end{align}

Since  each element in $L_{b}^1(\Omega, C(X))$ is measurable in $(\omega,x )$,  the $\sigma $-algebra   $\sigma(L_{b}^1(\Omega, C(X))$  generated by the functions of
$L_{b}^1(\Omega, C(X))$ is contained in  the product $\sigma$-algebra $\mathcal{B}(\Omega)\otimes \mathcal{B}(X)$. On the other hand, let  $A\in \mathcal{B}(\Omega)$ and $B$ be   a closed subset of $X$. We set $g_n(\omega,x):=\chi_A(\omega)\cdot b_n(x) \in L_{b}^1(\Omega, C(X))$, where $b_n(x)=1-\min\{nd(x,B),1\}$. Then  $\lim_{n\to \infty}g_n(\omega,x)=\chi_A(\omega)\chi_{B}(x)$  for all $(\omega,x)$. Therefore,  $A\times B \in \sigma(L_{b}(\Omega, C(X))) $ by the measurability of $g_n$.  This shows that  $\sigma(L_{b}(\Omega, C(X))= \mathcal{B}(\Omega)\otimes \mathcal{B}(X)=\mathcal{B}(\Omega \times X)$. Hence, $\mu_0$ is a Borel  probability measure on $\Omega \times X$.

Next, we   construct an invariant  probability measure on $\Omega \times X$  with  marginal  $\P$. Let  $f\in L_b^1(\Omega, C(X))$ and $g\in G$. Then  
\begin{align*}
||f\circ \Theta_g||=\int \sup_{x\in X}|f(g\omega, T_{g,\omega}x)| d\P(\omega)\leq \int ||f(g\omega)||_{\infty}d\P(\omega)=||f||,
\end{align*}
where we used the fact that $\mathbb{P}$ is $G$-invariant for the last equality. So $f\circ \Theta_g$ belongs to $ L_b^1(\Omega, C(X))$. By (\ref{equ 3.8})  
and the linearity of $L$,   for all $f\in L_b^1(\Omega, C(X))$ one has 
\begin{align}\label{equ 3.9}
\frac{L(\frac{1}{|F_n|}\sum_{g\in F_n}f\circ \Theta_g)}{L(1)}=\int fd\frac{1}{|F_n|}\sum_{g\in F_n}(\Theta_g)_{*}\mu_0,
\end{align}
where $\{F_n\}_{n\ge1}$ is a  F\o lner sequence of $G$. We set $$\mu_n= \frac{1}{|F_n|}\sum_{g\in F_n}(\Theta_g)_{*}\mu_0.$$ Then $\{\mu_n\}$ is a sequence of Borel probability measures  on $\Omega \times X$.  For any $A\subset \Omega \times X$,  we define
$$\mu_n^{*}(A)=\inf_{B\in \mathcal{B}(\Omega \times X), B\supset A}\mu_n(B).$$
By  Lemma \ref{lem 3.3}, $\{\mu_n^{*}\}_{n\geq 1}$  is a sequence of  Radon measures  on $\Omega \times X$.  Since   $\mu_n^{*}(K)=\mu_n(K)\leq 1$ for  each   compact set    $K$ of $\Omega \times X$,  then, by Lemma \ref{lem 3.2},   without loss of generality we  may  assume that  $\mu_{n}^{*}$  converges to a  Radon measure $\mu^*$ on $\Omega \times X$. Notice that  $\mu^*$ is Borel-regular. The measure $$\mu:=\mu^*|_{\mathcal{B}(\Omega \times X)},$$ which is a   restriction of $\mu^*$ on $\mathcal{B}(\Omega \times X)$, is a Borel measure on $\Omega \times X$.

We need to show  $\mu \in \M$.  Let $L^1(\Omega, \mathcal{B}(\Omega), \P)$ denote the  usual $L^1$-space with $L^1$-norm  $||\cdot||_{L^1}$. Given $f\in L^1(\Omega, \mathcal{B}(\Omega), \P)$,    by setting $\widetilde{f}|_{\{\omega\}\times X}=f(\omega)$ for  every $\omega$, this  yields  a function  $\widetilde{f} \in L_{\mathbb{P}}^1(\Omega, C(X))$. Clearly,  $\widetilde{\chi_A}=\chi_{A\times X}$ for any $A\in \mathcal{B}(\Omega)$.   
By $L^1$-mean ergodic theorem \cite[Theorem  4.23]{kl16}, one has 
\begin{align*}
||\widetilde{\frac{1}{|F_n|}\sum_{g\in F_n}{\chi_A}\circ g}-\P(A)||
=&\int |\frac{1}{|F_n|}\sum_{g\in F_n}{\chi_A}\circ g(\omega)-\P(A)|d \P(\omega)\\
=&||{\frac{1}{|F_n|}\sum_{g\in F_n}{\chi_A}\circ g}-\P(A)||_{L^{1}} \to 0, n\to \infty.
\end{align*} 
The continuity of $L$ implies that
\begin{align}\label{equ 3.10}
\lim_{n\to \infty}\frac{L(\frac{1}{|F_n|}\sum_{g\in F_n}\widetilde{\chi_A}\circ \Theta_g)}{L(1)}=\lim_{n\to \infty}\frac{\widetilde{L(\frac{1}{|F_n|}\sum_{g\in F_n}{\chi_A}\circ g)}}{L(1)}=\P(A).
\end{align}
Let $K$ be a compact subset of $\Omega$.  Since $\mu_n^*|_{\mathcal{B}(\Omega \times X)}=\mu_n$ for each $n$,  one has 
\begin{align}\label{equ 3.11}
\P(K)&  \overset{by~\text{(\ref{equ 3.10})}  }{=}\lim\limits_{n\to \infty}\frac{L(\frac{1}{|F_n|}\sum_{g\in F_n}\widetilde{\chi_K}\circ \Theta_g)}{L(1)} \overset{by~\text{(\ref{equ 3.9})}  }{=}\lim_{n \to \infty}\int \chi_{K\times X}d\mu_n^*\\
&= \limsup_{n \to \infty}\mu_n^*(K\times X)\overset{by~\text{ Lemma \ref{lem 3.1}}}{\leq} \mu^*(K\times X)=\mu(K\times X).\nonumber
\end{align}
By  \cite[Theorem B.13, p.425]{ew13},  $\mathbb{P}$ is regular, that is,  for any $A\in \mathcal{B}(\Omega)$
\begin{align}
\P(A)=\sup_{K\subset A,  K\text{ is compact}} \P(K)\leq \mu(A\times X),~ \text{by (\ref{equ 3.11})}.
\end{align}
The  similar arguments give  us $\P(A)\geq \mu(A\times X)$. Thus, $\P(A)=\mu(A\times X)$ for all $A\in \mathcal{B}(\Omega)$. Specially, we have $\mu(\Omega \times X)=1$. This shows $(\pi_{\Omega})_{*}{\mu}=\P$ and hence $\mu \in \PP$. 

The remaining is to show that $\mu$ is $\Theta_g$-invariant for all $g\in G$.
For  every $f\in C_c(\Omega \times X)$, one has 
\begin{align}\label{equ 3.12}
\lim_{n\to \infty}\frac{L(\frac{1}{|F_n|}\sum_{g\in F_n}f\circ \Theta_g)}{L(1)}&\overset{by~\text{(\ref{equ 3.9})}  }{=}\lim_{n\to \infty}\int fd\mu_n\\
&\overset{by~\text{ Lemma \ref{lem 3.3},(ii)}}{=}\lim_{n\to \infty}\int fd\mu_n^*\nonumber\\
&=\int fd{\mu^*}=\int fd\mu,  \nonumber
\end{align}
where the limits  exist since $\mu_n^{*} \to \mu$ as $n \to \infty$.
Let $f\in   L^1_b(\Omega, C(X))$   and $\gamma >0$. Choose $f_*\in C_c(\Omega \times X)$ such  that $||f-f_*||<\frac{L(1)\gamma}{3||L||}$ by Proposition \ref{prop 2.10},(3), where $||L||$ is the operator norm of $L$.  By (\ref{equ 3.12}), for sufficiently large $n$, we have
$$|\frac{L(\frac{1}{|F_n|}\sum_{g\in F_n}f_*\circ \Theta_g)}{L(1)}-\int f_*d\mu|<\frac{\gamma}{3}.$$
So
\begin{align*}
|\frac{L(\frac{1}{|F_n|}\sum_{g\in F_n}f\circ \Theta_g)}{L(1)}-\int fd\mu|
&\leq|\frac{L(\frac{1}{|F_n|}\sum_{g\in F_n}f\circ \Theta_g)}{L(1)}-\frac{L(\frac{1}{|F_n|}\sum_{g\in F_n}f_*\circ \Theta_g)}{L(1)}|+\\
&|\frac{L(\frac{1}{|F_n|}\sum_{g\in F_n}f_*\circ \Theta_g)}{L(1)}-\int f_*d\mu|+|\int f_*d\mu-\int fd\mu|\\
&<\gamma,
\end{align*}
where the first term is bounded by  $\sum_{g\in F_n}\frac{|L((f-g)\circ \Theta_g)|}{ L(1) |F_n|}\leq \sum_{g\in F_n}\frac{||L||\cdot ||f-f_*||}{ L(1) |F_n|}<\frac{\gamma}{3}$.
This shows that
\begin{align}\label{equ 2.13}
\lim_{n\to \infty}\frac{L(\frac{1}{|F_n|}\sum_{g\in F_n}f\circ \Theta_g)}{L(1)}=\int fd\mu
\end{align} 
holds for all $f\in   L^1_b(\Omega, C(X))$.  Let $h\in G$ and $f\in   L^1_b(\Omega, C(X))$. Replacing  $f$ by $f\circ \Theta_h$ in (\ref{equ 2.13}),  we have 
\begin{align*}
|\int f\circ \Theta_h d\mu-\int fd\mu|&=\lim\limits_{n \to \infty}\frac{1}{L(1)|F_n|}| \sum_{g\in hF_n\Delta F_n}  L(f\circ \Theta_g)|\\
&\leq \lim_{n \to \infty} \frac{ ||L||\cdot ||f||}{L(1)}\frac{|hF_n\Delta F_n|}{|F_n|}=0.
\end{align*}
Note that $L_b^1(\Omega, C(X))$ is dense in $L_{\mathbb{P}}^1(\Omega, C(X))$. A standard approximation approach shows that   $\int f\circ \Theta_h d\mu=\int fd\mu$  for all $f\in  L_{\mathbb{P}}^1(\Omega, C(X))$ and $h\in G$. This shows that $\mu \in \mathcal{M}_{\mathbb{P}}(\Omega \times X,G)$. 

Finally, we  show  that $\mu$  exactly satisfies  the inequality (\ref{equ 3.5}):   $$\inf_{f\in \widetilde{\mathcal{A}}_c}\int fd\mu+\int c d\mu+\epsilon>0.$$  
For   each $f\in \widetilde{\mathcal{A}}_c$ and $g\in G$, one has  $f\circ \Theta_g \in \widetilde{\mathcal{A}}_b$ by the  semi-cohomology property of $\Gamma$. It follows  from (\ref{equ 3.12}) that
\begin{align}\label{equ 3.13}
\int fd\mu= \lim_{n\to \infty}\frac{1}{|F_n|}\sum_{g\in F_n}\frac{L(f\circ \Theta_g)}{L(1)}\geq \frac{1}{L(1)}\inf_{f\in \widetilde{\mathcal{A}}_b}L(f),
\end{align}
which implies that  $s(\mu)\geq\frac{1}{L(1)}\inf_{g\in \widetilde{\mathcal{A}}_b}L(g)$ by Proposition \ref{lem 2.6}.  Therefore,   one has 
\begin{align*}
s(\mu)+\int cd\mu+\epsilon &\geq\frac{1}{L(1)} \inf_{f\in \widetilde{\mathcal{A}}_b+\frac{\epsilon}{2}}L(f)+ \frac{L(c)}{L(1)}+\frac{\epsilon}{2}\\
&= \frac{1}{L(1)}(\inf_{f\in \widetilde{\mathcal{A}}_b+\frac{\epsilon}{2}}{L(f)}+L(c)) +\frac{\epsilon}{2}>0, \text{by (\ref{equ 3.7})}.
\end{align*}

It is easy to see that $\M$ is  a  closed subset of $\PP$ in the weak$^{*}$-topology.   Since  $s(\mu)$ is  an  upper  semi-continuous and concave  function on $\M$, then the set 
$$M_{max}(T,\Gamma):=\{\mu\in \M: \Gamma(0)=s(\mu)\}$$
is a non-empty  compact convex subset of $\M$. Thus, the supremum can be  attained for some invariant measures.
\end{proof}

As we have  mentioned, if $\Omega=\{\omega\}$ is  a   single point and $\mathbb{P}$ is the  Dirac measure  $\delta_{\omega}$ at $\omega$, then the RDS  is reduced to the $G$-system  $(X,G)$,   and $L_{\mathbb{P}}^1(\Omega, C(X))$ coincides with $C(X)$ equipped with the supremum norm. In this case,  condition (4) in Definition \ref{Def3.1} can be removed  since it can be deduced from conditions  (1) and  (2).  More precisely, for any $f_1, f_2 \in C(X)$, we have
$$\Gamma(f_1)\leq \Gamma(f_2+||f_1-f_2||_{\infty})=\Gamma(f_2)+||f_1-f_2||_{\infty}.$$
Exchanging the role of $f_1$ and $f_2$, we get the converse inequality.

 Let  $G=\mathbb{Z}_{+}^k$, or $G$ be an amenable group continuously acting on a compact metric space $X$  via continuous self-maps $\{T_g: g\in G\}$. If the pressure function  $\Gamma: C(X)\rightarrow \mathbb{R}$ satisfies the following properties:

\begin{enumerate}
\item Monotonicity:  $f\leq g \Longrightarrow$ $\Gamma(f)\leq \Gamma(g)$ $ \forall f,g \in  C(X)$.
\item  Translation invariance: $\Gamma(f+c)= \Gamma(f)+c $ $ \forall f \in  C(X)$ and $ c\in \mathbb{R}$.
\item Convexity:    $\Gamma(pf+(1-p)g)\leq p\Gamma(f)+(1-p)\Gamma(g)$  $\forall f, g\in C(X)$ and $p\in [0,1]$.
\item Semi-cohomology:  $\Gamma(f\circ T_g)\leq \Gamma (f)$ $ \forall f \in  C(X)$ and $g \in G$,
\end{enumerate}
then, by  Theorem \ref{thm 1.1}, we   obtain the following variational principle  for pressure functions in the context of $G$-actions.
\begin{thm}\label{thm 3.12}
Let $G=\mathbb{Z}_{+}^{k}$, $k\geq 1$, or $G$ be an  amenable group continuously acting on a compact metric space $(X,d)$.  If $\Gamma$ is a  pressure  function on $C(X)$, then 
$$ \Gamma(0)=\max_{\mu \in \mathcal{M}(X,G)}s(\mu),$$
where  
 \begin{align*}
s(\mu)&=\inf\{\Gamma(f)-\int fd\mu:f\in C(X)\}\\
&= \inf_{f\in \mathcal{A}}\int fd\mu,
\end{align*}
and $\mathcal{A}=\{f\in C(X):\Gamma(-f)=0\}$. Furthermore, $s(\mu)$ is a concave upper semi-continuous function on the set of $G$-invariant  Borel probability measures $\mathcal{M}(X,G)$.
\end{thm}

\section{Some applications of the main result}\label{sec 4}

In this section, we present several applications of Theorem \ref{thm 1.1}. The main results of this section are Theorems A-D, which corresponds to the answer to  \emph{Questions 1-3}.

\subsection{Application to polynomial topological entropy of  zero entropy systems}

Although the zero entropy systems are commonly regarded as a class of ``simple'' systems with lower topological complexity, it can exhibit intricate dynamical behaviors that can be detected by  the polynomial topological entropy.

We first review the definition of polynomial  entropy, which is defined by considering the  polynomial growth rate of  distinguishable orbits \cite{kt97}. Let a pair $(X,T)$ denote a  topological dynamical system (TDS for short), where  $T: X\rightarrow X$ is a  continuous self-map  on a compact  metric space $(X,d)$.  

Given a non-empty subset $Z \subset X$, $F\subset Z$ is called a $(d,\epsilon)$-separated set of $Z$ if any distinct $x,y \in F$, one has $d(x,y)>\epsilon$. Denote the largest  cardinality  of $(d,\epsilon)$-separated set of $Z$ by $s(Z,d,\epsilon)$.  The  Bowen metric  $d_n$ on $X$  is defined as 
$$d_n(x,y)=\max_{0\leq j \leq n-1}d(T^{j}(x), T^{j}(y))$$
for all $x,y \in X$.

Let $f\in C(X)$ and  set $S_{n}f(x):={\sum_ {j=0}^{\lceil \log n\rceil}f(T^{j}x)}$, where $\lceil  u \rceil$ is  the smallest integer greater than $u$. 
 
\begin{df}
 We define the \emph{polynomial  topological pressure of $f$ on $X$} as
$$P_{pol}(T,f)=\lim\limits_{\epsilon \to 0}\limsup_{n \to \infty}\frac{1}{\log n}\log \sup_{E\subset X}\{\sum_{x\in E}e^{S_{n}f(x)}\},$$
where the supremum ranges over  the set of $(d_n,\epsilon)$-separated sets of $X$.
\end{df}

We set $h_{pol}(T,X):=P_{pol}(T,0)$,  and call  $h_{pol}(T,X)$ the \emph{polynomial topological  entropy of $X$} \cite{kt97}.

Using the idea in Section \ref{sec 3},  we introduce the following:

\begin{df}
Let  $(X,T)$ be a TDS. For every $\mu \in \mathcal{M}(X,T)$, we define the polynomial measure-theoretic  entropy of $\mu$ as
$$h_{\mu}^{pol}(T)=\inf_{f\in C(X)}\{P_{pol}(T,f)-\int fd\mu\}.$$
\end{df}

\begin{prop}
$P_{pol}(T,\cdot): C(X)\rightarrow \mathbb{R}$ is a pressure functions on $C(X)$.
\end{prop} 
\begin{proof}
The monotonicity and translation invariance properties of $P_{pol}(T,\cdot)$ are clear. Let  $p\in [0,1]$ and $f,g\in C(X)$. If $E$ is a $(d_{n},\epsilon)$-separated set of $X$,  then, applying H\"{o}lder's inequality,  we have
\begin{align*}
\sum_{x\in E}e^{pS_{n}f(x)+(1-p)S_{n}g(x)}\leq (\sum_{x\in E}e^{S_{n}f( x)})^p(\sum_{x\in E}e^{S_{n}g(x)})^{(1-p)},
\end{align*}
which yields that  $P_{pol}(T,pf+(1-p)g)\leq pP_{pol}(T,f)+(1-p)P_{pol}(T,g)$. 

Let $f\in C(X)$. If $E$ is a $(d_{n},\epsilon)$-separated set of $X$, then for any $x\in E$, $$S_{n}(f\circ T)(x):={\sum_ {j=1}^{\lceil \log n\rceil +1}f(T^{j}x)}=S_{n}f(x)+f(T^{\lceil \log n\rceil})-f(x).$$  Notice that $|f(T^{\lceil \log n\rceil})-f(x)|\leq 2||f||_{\infty}$. This implies that   $P_{pol}(T,f\circ T)=P_{pol}(T,f)$. 
\end{proof} 

Thus, by Theorem \ref{thm 3.12}, we  obtain a variational principle for polynomial topological entropy, which gives an answer to \emph{Question 1}.

\begin{thma}\label{theorema}
Let $(X,T)$ be a TDS. 
 If $h_{pol}(T,X)<\infty$, then
\begin{align*}
h_{pol}(T,X)=\max_{\mu \in \mathcal{M}(X,T)} h_{\mu}^{pol}(T).
\end{align*}
\end{thma}

\subsection{Application to mean dimensions of infinite entropy systems}

To start with, we introduce the mean dimensions with potential of $G$-actions.

Let $(X,G)$ be a $G$-system with a metric $d$ on $X$. Denote  the set of non-empty finite subsets of $G$ by $\mathcal{F}(G)$.
Let  $F\in \mathcal{F}(G)$.  The Bowen metric w.r.t. $F$ is  given by
$$d_F(x,y)=\max_{g\in F}d(gx,gy).$$

Let $\epsilon >0$  and $f\in C(X)$.  We put
$$P_F(X,f,d,\epsilon)=\sup\{\sum_{x\in E}e^{S_Ff(x)}: E~\mbox{is a}~ (d_F,\epsilon)\mbox{-separated set of}~X\}.$$
\begin{df}
The metric mean dimension of $X$ with potential $f$ is defined by 
$${\rm \overline{mdim}}_M(G,X,f,d)=\limsup_{\epsilon \to 0}\frac{1}{\logf}\limsup_{n \to \infty}\frac{1}{|F_n|}\log P_{F_n}(X,f\cdot \logf,d,\epsilon),$$
where $\{F_n\}$ is a F\o lner sequence of $G$.
\end{df}

It is well-known that the metric mean dimension of $X$ with potential $f$ is independent of the choice of F\o lner sequence  $\{F_n\}$ of $G$. Letting $f=0$ and ${\rm \overline{mdim}}_M(G,X,d):={\rm \overline{mdim}}_M(G,X,0,d)$, we call ${\rm \overline{mdim}}_M(G,X,d)$ the \emph{metric mean dimension  of $X$}.

Let $\mathcal{C}_X^o$ denote the set of  all finite open covers of $X$. Given an  open cover $\mathcal{U} \in \mathcal{C}_X^o$,  we put ${\rm ord}_x(\mathcal{U})=\sum_{U \in \mathcal{U}}\chi_{U}(x)-1$ for every $x\in X$ and ${\rm ord}(\mathcal{V}):=\max_{x\in X}{\rm ord}_x(\mathcal{V})$. The order of $\UU$ is defined as
$$\mathcal{D}(\mathcal{U})=\min_{\mathcal{V}\succ\mathcal{U}}\rm{ord}(\mathcal{V}),$$
where the infimum is taken over all finite open covers  of $X$ that refines $\mathcal{U}$. 

For  $F\in \mathcal{F}(G)$, we define $\mathcal{U}_{F}=\vee_{g\in F} g^{-1}\mathcal{U}$. Then it is easy to verify that the function $$F\in \mathcal{F}(G) \mapsto \mathcal{D}(\mathcal{U}_F)$$ satisfies the   conditions of the well-known Ornstein-Weiss theorem \cite{ow87}. Hence,  the limit
$${\rm mdim}(G,\mathcal{U}):=\lim_{n\to\infty}\frac{\mathcal{D}(\mathcal{U}_{F_n})}{|F_n|}$$
exists and  does not depend on the choice of  F\o lner sequences $\{F_n\}_{n\ge1}$  of $G$. 

The  \emph{mean dimension} of $X$ is defined by
$${\rm mdim}(X,G)=\sup_{\mathcal{U} \in \mathcal{C}_X^o}{\rm mdim}(G,\mathcal{U}).$$


Fix a  F\o lner sequence $\{F_n\}_{n\ge1}$  of $G$ and a continuous potential $f:X\rightarrow \mathbb{R}$.  We define  $S_{F_n}(x)=\sum_{g\in F_n}f(gx)$   to be  the Birkhoff sum of $x$ w.r.t. $F_n$.  We set 
$${\rm mdim}(G,f,\mathcal{U};\{F_n\}):=\limsup_{n \to \infty}\frac{1}{|F_n|}\sup_{\mathcal{V}\in \mathcal{P}_n(\UU)}\{{\mathcal{D}}(S_{F_n}f,\mathcal{V})\},$$ 
where  $\mathcal{P}_n(\UU)$ is the set of finite open covers $\VV$ that realize the minimum  for $\mathcal{D}(\mathcal{U}_{F_n})$,  and $ {\mathcal{D}}(S_{F_n}f,\mathcal{V}):=\max_{x\in X}\{{\rm ord}_x(\mathcal{V})+S_{F_n}f(x)\}$.
\begin{df}
We define the \emph{mean dimension of $X$ with potential $f$ w.r.t. $\{F_n\}$} as
$${\rm mdim}(X,G,f;\{F_n\})=\sup_{\mathcal{U} \in \mathcal{C}_X^o}{\rm mdim}(G,f,\mathcal{U};\{F_n\}).$$
\end{df}
Clearly,  ${\rm mdim}(X,G)={\rm mdim}(X,G,0;\{F_n\})$  if $f$ is a zero potential.

\begin{prop}
Let $(X,G)$ be a  $G$-system with a metric $d$. Then  for every  F\o lner sequence $\{F_n\}_{n\ge1}$  of $G$, $${\rm mdim}(G,\cdot,\mathcal{U};\{F_n\}):C(X)\rightarrow \mathbb{R}$$ and  ${\rm \overline{mdim}}_M(G,X,\cdot,d):C(X)\rightarrow \mathbb{R}$ are two  pressure functions.
\end{prop}
\begin{proof}
For every  F\o lner sequence  $\{F_n\}$ of $G$,  it is clear that  ${\rm mdim}(G,\cdot,\mathcal{U};\{F_n\})$ satisfies monotonicity, translation invariance and  convexity. The remaining  it  to check  the cohomology.

 Fix $g\in G$ and $f\in C(X)$.  Then for every $\mathcal{V} \in \mathcal{P}_n(\UU)$,
\begin{align*}
{\mathcal{D}}(S_{F_n}(f\circ g),\mathcal{V})&=\max_{x\in X}\{{\rm ord}_x(\mathcal{V})+S_{gF_n}{f}(x)\}\\
&\leq {\mathcal{D}}(S_{F_n}f,\mathcal{V})+  |gF_n\triangle F_n|\cdot ||f||_{\infty}.
\end{align*}  
The  amenability of F\o lner  sequence   implies that  $${\rm mdim}(G,f\circ g,\mathcal{U};\{F_n\})\leq {\rm mdim}(G,f,\mathcal{U};\{F_n\})$$ for all $g\in G$. Thus, we conclude that   ${\rm mdim}(G,\cdot,\mathcal{U};\{F_n\}):C(X)\rightarrow \mathbb{R}$ is a pressure function. The same arguments show that  ${\rm \overline{mdim}}_M(G,X,\cdot,d):C(X)\rightarrow \mathbb{R}$ is also a pressure function.
\end{proof}

\begin{df}
Let $(X,G)$ be a $G$-system, and  $\{F_n\}$ be a  F\o lner sequence  of $G$. For every $\mu \in \mathcal{M}(X,G)$, 

$(1)$ we define the amenable measure-theoretic mean dimension of $\mu$ w.r.t. $\{F_n\}$ as
$${\rm mdim}_{\mu}(X,G;\{F_n\})=\inf_{f\in C(X)}\{{\rm mdim}(X,G,f;\{F_n\})-\int fd\mu\};$$

$(2)$ we define the  amenable measure-theoretic metric mean dimension of $\mu$  as $${\rm \overline{mdim}}_M(G,\mu,X,d)=\inf_{f\in C(X)}\{{\rm \overline{mdim}}_M(G,X,f,d)-\int fd\mu\}.$$ 
\end{df}

Applying the  Theorem \ref{thm 3.12} to the two types of pressure functions, we obtain the following new variational principles for mean dimensions  whose forms are more close to the classical ones, which do not assume  the marker property for dynamical systems and give an answer to \emph{Question 2}.

\begin{thmb}\label{thm 4.3}
Let $(X,G)$ be a $G$-system. 

$(1)$ If ${\rm mdim}(X,G)<\infty$, then for every  F\o lner sequence  $\{F_n\}$ of $G$
\begin{align*}
{\rm mdim}(X,G)=\max_{\mu \in \mathcal{M}(X,G)}{\rm mdim}_{\mu}(X,G;\{F_n\}).
\end{align*}

$(2)$ If ${\rm \overline{mdim}}_M(G,X,d)<\infty$, then
\begin{align*}
{\rm \overline{mdim}}_M(G,X,d)=\max_{\mu \in \mathcal{M}(X,G)} {\rm \overline{mdim}}_M(G,\mu,X,d).
\end{align*}
\end{thmb}

\subsection{Application to preimage entropy-like quantities of non-invertible RDSs}

\subsubsection{Preimage entropy-like quantities of non-invertible RDSs} 
We assume that $G=\mathbb{Z_+}$. Recall that a non-invertible RDS over $(\Omega,\mathcal{F},\mathbb{P},\theta)$ is generated by mappings $T_{\omega}:X\rightarrow X$ with  iterates  
\begin{align*}
T_{\omega}^n=
\begin{cases}
T_{\theta^{n-1}\omega}\circ T_{\theta^{n-2}\omega}\circ\cdots \circ T_{\omega},  &\mbox{for}~n\geq 1\\
id,&\mbox{for}~n=0
\end{cases}
\end{align*}      
such that $(\omega,x)\mapsto T_{\omega}x$ is measurable and $x\mapsto T_{\omega}x$ is continuous for all $\omega$.

In the following, we first  recall the  precise definitions  of three types of preimage  pressures(and  entropies) for non-invertible  RDSs:  pointwise preimage entropies  and topological preimage entropy.

Given $\omega \in \Omega$, $n\in \mathbb{N}$, for any $x,y \in X$, we define the Bowen  metric
$$d_n^{\omega}(x,y)=\max_{0\leq j  \leq n-1}d(T_{\omega}^{j}x,T_{\omega}^{j}y).$$ 
Letting $f\in L_{\mathbb{P}}^1(\Omega, C(X))$,  we  set 
$$S_nf(\omega,x)=\sum_{j=0}^{n-1}f\circ\Theta^j(\omega,x)=\sum_{j=0}^{n-1}f(\theta^j\omega,T_{\omega}^jx).$$

Let $E$ be a non-empty set of $X$. A set $F\subset E$ is an  \emph{$(\omega,n,\epsilon)$-separated set} if   $d_n^{\omega}(x,y)>\epsilon$ for any  $x,y \in F$ with $x\not=y$. Let  $s_n(T,E,\omega,\epsilon)$  denote the   maximal  cardinality  of {$(\omega,n,\epsilon)$-separated sets} of $E$.

Topological preimage entropy was  firstly introduced by Cheng and Newhouse \cite{cn05} and was extended to  topological preimage pressure  by Zeng  et al. \cite{zyz07}, which  were  further investigated by  the   several authors  for   non-invertible RDSs  \cite{z07,mc09,z1109}. We  define  
\begin{align*}
P_{pre,n}(T,f,\omega,\epsilon,k)
=\sup_{x\in X}\sup\{\sum_{y\in  F}e^{S_nf(\omega,y)}: F~ \text{is an}~(\omega, n, \epsilon)\text{-separated set of}~ T_\omega^{-k}x\}.
\end{align*}
Since the quantity $P_{pre,n}(T,f,\omega,\epsilon,k)$  is measurable in $\omega$ \cite[Lemma 2.1]{z1109},  we set
$$P_{pre}(T,f,\epsilon)=\limsup_{n \to \infty}\frac{1}{n}\sup_{k\geq n}\int \log P_{pre,n}(T,f,\omega,\epsilon,k)d\P(\omega).$$
The \emph{random topological preimage pressure of $f$}  \cite{z1109} is defined by
\begin{align*}
P_{pre}^{*}(T,f)&=\lim\limits_{\epsilon \to 0}P_{pre}(T,f,\epsilon).
\end{align*}

If $f=0$, we let $h_{pre}^{*}(T):=P_{pre}^{*}(T,0)$,  and call $h_{pre}^{*}(T)$  \emph{random  topological preimage entropy of $X$} \cite{z07,mc09}.

A pointwise approach  to preimage entropy  is due to Hurley \cite{hur95} for $\mathbb{Z}_{+}$-actions. See also  \cite{wz21} for the extensions of  pointwise  preimage entropies in  $\mathbb{Z}_+$-actions and in  non-invertible RDSs \cite{lwz20,wwz23}. We put
\begin{align*}
P_{pre,n}(T,f,\omega,x,\epsilon)
=\sup_{}\{\sum_{y\in  E}e^{S_nf(\omega,y)}: E~ \text{is an}~(\omega, n, \epsilon)\text{-separated set of}~ T_\omega^{-n}x\}.
\end{align*}
In  \cite{wwz23},  Wang, Wu and Zhu  showed  that  both  the quantities $P_{pre,n}(T,f,\omega,x,\epsilon)$ and $\sup_{x\in X}P_{pre,n}(T,f,\omega,x,\epsilon)$  are measurable in $\omega$. This allows us to  define 
\begin{align*}
P_{pre,m}(T,f,\epsilon)&=\limsup_{n \to \infty}\frac{1}{n}\int \log \sup_{x\in X}P_{pre,n}(T,f,\omega,x,\epsilon)d\P(\omega),\\
P_{pre,p}(T,f,\epsilon)&= \sup_{x\in X}\limsup_{n \to \infty}\frac{1}{n}\int \log P_{pre,n}(T,f,\omega,x,\epsilon)d\P(\omega).
\end{align*}
The \emph{random pointwise preimage pressures of $f$}  are defined by
\begin{align*}
P_{pre,m}^{*}(T,f)&=\lim\limits_{\epsilon \to 0}P_{pre,m}(T,f,\epsilon),\\
P_{pre,p}^{*}(T,f)&=\lim\limits_{\epsilon \to 0}P_{pre,p}(T,f,\epsilon).
\end{align*}

If $f=0$, we let $h_{pre,m}^{*}(T):=P_{pre,m}^{*}(T,0)$ and   $h_{pre,p}^{*}(T):=P_{pre,p}^{*}(T,0)$, and  call $h_{pre,m}^{*}(T)$, $ h_{pre,p}^{*}(T)$  \emph{random  topological preimage entropies of $X$}, respectively. 

For any $f\in L_{\mathbb{P}}^1(\Omega, C(X))$, it   is clear that
\begin{align*}
&P_{pre,p}^{*}(T,f)\leq P_{pre,m}^{*}(T,f)\leq P_{pre}^{*}(T,f),
\end{align*}
and  the three types of preimage topological pressures  are independent of the choice of the  compatible  metrics on $X$.

To distinguish these non-invertible dynamical systems with infinite preimage  entropies, inspired by \cite{lw00,t20} we introduce the corresponding preimage metric mean dimension  (with potential) for  the  above three types of infinite preimage entropies. This is  realized  by capturing the ``information"  concerning  how fast  the corresponding preimage $\epsilon$-entropies  diverges to  $\infty$ as $\epsilon  \to 0$.

\begin{df} \label{df 4.10}
Let $T=(T_{\omega})$ be a RDS over the measure-preserving system $(\Omega,\mathcal{F},\mathbb{P},\theta)$, and let $f \in  L_{\mathbb{P}}^1(\Omega, C(X))$.

$(1)$  We define the upper  preimage metric mean dimension of $X$ with potential $f$ as 
$$\mathbb{E}{\overline{\rm mdim}}_{pre}(T,f,d):=\limsup_{\epsilon \to 0}\frac{P_{pre}(T,f\cdot \logf,\epsilon)}{\logf}.$$

$(2)$  We define the upper  pointwise preimage metric mean dimension of $X$ with potential $f$ as 
\begin{align*}
\mathbb{E}{\overline{\rm mdim}}_{pre,s}(T,f,d):&=\limsup_{\epsilon \to 0}\frac{P_{pre,s}(T,f\cdot \logf,\epsilon)}{\logf},
\end{align*}
where $s\in \{m,p\}$.
\end{df}

In particular,  for the zero potential case we let $\mathbb{E}{\overline{\rm mdim}}_{pre}(T,X,d):= \mathbb{E}{\overline{\rm mdim}}_{pre}(T,0,d)$, $\mathbb{E}{\overline{\rm mdim}}_{pre,m}(T,X,d):=\mathbb{E}{\overline{\rm mdim}}_{pre,m}(T,0,d)$, and $\mathbb{E}{\overline{\rm mdim}}_{pre,p}(T,X,d):=\mathbb{E}{\overline{\rm mdim}}_{pre,p}(T,0,d)$, and call   $\mathbb{E}{\overline{\rm mdim}}_{pre}(T,X,d)$, $\mathbb{E}{\overline{\rm mdim}}_{pre,m}(T,X,d)$, and  $\mathbb{E}{\overline{\rm mdim}}_{pre,p}(T,X,d)$ the random  preimage metric mean dimension of $X$, random pointwise metric mean dimension of $X$, respectively.

For any $f\in L_{\mathbb{P}}^1(\Omega, C(X))$, it is clear from the definitions  that
\begin{align*}
&P_{pre,p}(T,f)\leq P_{pre,m}(T,f)\leq P_{pre}(T,f),\\
&\mathbb{E}{\overline{\rm mdim}}_{pre, p}(T,f,d)\leq \mathbb{E}{\overline{\rm mdim}}_{pre, m}(T,f,d) \leq  \mathbb{E}{\overline{\rm mdim}}_{ pre}(T,f,d).
\end{align*}

\subsubsection{Variational principles of  preimage entropy-like quantities}
Next we  show the aforementioned preimage  pressures-like quantities   are  random pressure functions on   $L_{\mathbb{P}}^1(\Omega, C(X))$.

\begin{prop}\label{prop 4.11}
Let $T=(T_{\omega})$ be a RDS over the measure-preserving system $(\Omega,\mathcal{F},\mathbb{P},\theta)$. 
Let $f,g\in L_{\mathbb{P}}^1(\Omega, C(X))$. Then  $P_{pre}^{*}(T,\cdot)$ satisfies  the  following statements: 
\begin{enumerate}
\item  $ h_{pre}^{*}(T)-||f|| \leq P_{pre}^{*}(T,f) \leq h_{pre}^{*}(T)+||f|| $.
\item  $P_{pre}^{*}(T,\cdot): L_{\mathbb{P}}^1(\Omega, C(X))\longrightarrow \R\cup\{\infty\}$ is either finite value or constantly $\infty$.
\item (monotonicity) If  $f\leq g$, then  $P_{pre}^{*}(T,f)\leq P_{pre}^{*}(T,g)$.\\
\item (translation invariance) $ P_{pre}^{*}(T,f+c)=P_{pre}^{*}(T,f)+c$ for  any $c\in \R$.

\item (Lipschitz and convexity) If $h_{pre}^{*}(T)<\infty$, then 
$$|P_{pre}^{*}(T,f)-P_{pre}^{*}(T,g)|\leq ||f-g||$$
and $P_{pre}^{*}(T,\cdot)$ is convex on $L_{\mathbb{P}}^1(\Omega, C(X))$. 
\item (cohomology) $P_{pre}^{*}(T,f+g\circ \Theta-g)= P_{pre}^{*}(T,f)$.

\end{enumerate}
Furthermore,  the above properties are  also valid for 
$ P_{pre,p}^{*}(T,\cdot)$,
$P_{pre,m}^{*}(T,\cdot)$.
\end{prop}

\begin{proof}
These properties  can be  proved by  modifying the proof of  Proposition \ref{Prop 3.2}.  Due to the  completeness, we  give a sketch for $P_{pre}^{*}(T,\cdot)$.

(1-2).  It   follows from the inequality
\begin{align*}
e^{\sum_{j=0}^{n-1}-||f(\theta^j\omega)||_{\infty}} P_{pre,n}(T,0,\omega,\epsilon,k)&\leq P_{pre,n}(T,f,\omega,\epsilon,k)\\
&\leq e^{\sum_{j=0}^{n-1}||f(\theta^j\omega)||_{\infty}} P_{pre,n}(T,0,\omega,\epsilon,k).
\end{align*}
for any $k\geq n$ and $\omega \in \Omega$.

(5).  Let $0<\epsilon <1, k\geq n$  and $\omega \in \Omega$. Then 
\begin{align*}
P_{pre,n}(T,f,\omega,\epsilon,k)
\leq e^{\sum_{j=0}^{n-1}||(f-g)(\theta^j\omega)||_{\infty}} P_{pre,n}(T,g,\omega,\epsilon,k), 
\end{align*}
which implies that  $P_{pre}^{*}(T,f)\leq P_{pre}^{*}(T,g)+||f-g||.$
Exchanging the role of $f$ and $g$,  one gets  $$P_{pre}^{*}(T,g)\leq P_{pre}^{*}(T,f)+||f-g||.$$

Let $p\in[0,1]$, and let $E$ be an  $(\omega, n,\epsilon)$-separated set of $T_{\omega}^{-k}x$. Applying H\"{o}lder's inequality, we have
\begin{align*}
\sum_{y\in E}e^{pS_nf(\omega, y)+(1-p)S_ng(\omega, y)}\leq (\sum_{y\in E}e^{S_nf(\omega, y)})^p(\sum_{y\in E}e^{S_ng(\omega, y)})^{(1-p)},
\end{align*}
which yields that  $P_{pre}^{*}(T,pf+(1-p)g)\leq pP_{pre}^{*}(T,f)+(1-p)P_{pre}^{*}(T,g)$.

(6) Let  $F$ be an  $(\omega, n,\epsilon)$-separated set of $T_{\omega}^{-k}x$. Then
$$\sum_{y\in F}e^{S_n(f+g\circ \Theta -g)(\omega, y)}= \sum_{y\in F}e^{S_nf(\omega, y)+g(\theta^n\omega,T_\omega^ny)-g(\omega,y)}.$$
Then the desired inequality holds by considering the following inequality
\begin{align*}
&\sum_{y\in F}e^{S_nf(\omega, y)-||(g(\theta^n\omega)||_{\infty}-||(g(\omega)||_{\infty}}\\
\leq&\sum_{y\in F}e^{S_n(f+g\circ \Theta -g)(\omega, y)}\\
\leq &\sum_{y\in F}e^{S_nf(\omega, y)+||(g(\theta^n\omega)||_{\infty}+||(g(\omega)||_{\infty}}.
\end{align*}
\end{proof}

With the above notions, we introduce the preimage metric mean dimensions of invariant measures.

\begin{df}\label{df 4.12}
Let $T=(T_{\omega})$ be a RDS over the measure-preserving system with $h_{pre}^{*}(T)<\infty$. For any $\mu \in \PP$, we  respectively define  the \emph{measure-theoretic preimage entropy, measure-theoretic point-wise preimage entropies of $\mu$}
\begin{align*}
h_{pre,\mu}^{*}(T)&=\inf\{P_{pre}^{*}(T,f)-\int fd\mu:f\in L_{\mathbb{P}}^1(\Omega, C(X))\},\\
h_{s,\mu}^{*}(T)&=\inf\{P_{pre,s}^{*}(T,f)-\int fd\mu:f\in L_{\mathbb{P}}^1(\Omega, C(X))\},
\end{align*}
 where $s\in \{p,m\}$.
\end{df}

Based on  the Proposition \ref{prop 4.11} and  Theorem \ref{thm 1.1}, one can   formulate the  following new variational principles for random   preimage entropy-like quantities without  requiring the uniform separation of   preimages  for the dynamical systems.

\begin{thmc}\label{thm 4.13}
Let $\Omega$ be a locally compact and separable metric space  with Borel $\sigma$-algebra $\mathcal{B}(\Omega)$. Let $T=(T_{\omega})$ be a   random $\mathbb{Z}_+$ dynamical system  over  an ergodic measure-preserving system $(\Omega, \mathcal{B}(\Omega),\P, \theta)$. If  $h_{pre}^{*}(T)<\infty$, then 
\begin{align*}
h_{pre}^{*}(T)&=\max_{\mu \in \M}h_{pre,\mu}^{*}(T),\\
h_{pre,s}^{*}(T)&=\max_{\mu \in \M}h_{s,\mu}^{*}(T),
\end{align*}
where $s\in \{p,m\}$    
\end{thmc}

\begin{rem}
 When $\Omega$ is a single point,  for  systems  of $\mathbb{Z}_{+}$-actions  with the property of uniform separation of preimages, Li, Wu and Zhu \cite[Theorem A, Proposition 4.2]{lwz20} showed that 
\begin{align*}
{h}_{m,\mu}(T)&=\inf\{{P}_{pre,m}(T,f)-\int fd\mu:f\in C(X)\},\\
&=\inf\{{P}_{pre,p}(T,f)-\int fd\mu:f\in  C(X) \},
\end{align*}
where ${P}_{pre,m}(T,f),$ ${P}_{pre,p}(T,f)$ are    pointwise preimage pressures of $f$ respectively. 
Hence,  the  Definition \ref{df 4.12}   is  compatible with  the $\mathbb{Z}_{+}$-actions,  which  provides a strategy to extend the variational principles  of pointwise preimage entropies  to any  dynamical systems, and establishes  a new  Cheng-Newhouse's variational principle for topological  preimage  entropy. 
\end{rem}

Since the property of uniform separation of preimages on dynamical systems implies the systems have infinite preimage topological entropies, so these systems  have zero preimage metric mean dimensions. Therefore, the previous techniques for establishing the variational principles for preimage topological entropy-like quantities \cite{cn05,wz21}  fail to apply to preimage  metric mean dimensions in such cases. The  convex approach that we developed  works for this case.  One can show that preimage metric mean dimensions with potential are (random) pressure functions. Then we use Theorem \ref{thm 1.1} to get the  variational principles for preimage metric mean dimension:

\begin{thmd}\label{thm 4.15}
Let $\Omega$ be a locally compact and separable metric space  with Borel $\sigma$-algebra $\mathcal{B}(\Omega)$. Let $T=(T_{\omega})$ be a   random $\mathbb{Z}_+$-dynamical system  over  an ergodic measure-preserving system $(\Omega, \mathcal{B}(\Omega),\P, \theta)$. If  $\mathbb{E}{\overline{\rm mdim}}_{pre}(T,X,d)<\infty$, then 
\begin{align*}
\mathbb{E}{\overline{\rm mdim}}_{pre}(T,X,d)&=\max_{\mu \in \M} \mathbb{E}{\overline{\rm mdim}}_{pre,\mu}(T),\\
\mathbb{E}{\overline{\rm mdim}}_{pre,s}(T,X,d)&=\max_{\mu \in \M}\mathbb{E}{\overline{\rm mdim}}_{s,\mu}(T),
\end{align*}
where $s\in \{p,m\}$  and   
\begin{align*}
 \mathbb{E}{\overline{\rm mdim}}_{pre,\mu}(T)&=\inf\{\mathbb{E}{\overline{\rm mdim}}_{pre}(T,f,d)-\int fd\mu:f\in L_{\mathbb{P}}^1(\Omega, C(X))\},\\
\mathbb{E}{\overline{\rm mdim}}_{s,\mu}(T)&=\inf\{\mathbb{E}{\overline{\rm mdim}}_{pre,s}(T,f,d)-\int fd\mu:f\in L_{\mathbb{P}}^1(\Omega, C(X))\}.
\end{align*} 
\end{thmd}

\section{Appendix}
In this section, we show  that the random topological pressure   is independent of  the choice of the  F\o lner sequences of amenable groups.

A set function $h:\FF(G)\rightarrow \R$  is said  to be \\
(1)  monotone if $h(E)\le h(F)$ for  $\forall E,F\in\FF(G)$ with $E\subset F$;\\
(2) $G$-invariant if $h(Eg)= h(E)$  for $\forall g\in G$, $\forall E\in\FF(G)$; \\
(3) sub-additive if $h(E\cup F)\le h(E)+h(F)$   for any disjoint $E,F\in\FF(G)$.

The following lemma is the well-known Ornstein-Weiss Theorem, which is an amenable version of  Fekete's lemma.
\begin{lem}\cite{ow87,gromov,lw00}\label{lem 5.1}
	Let $G$ be a countable amenable group. Let  $h:\FF(G)\rightarrow \R$  be a monotone $G$-invariant sub-additive function. Then there exists $\lambda\in [-\infty,\infty)$ depending on $G$ and $h$  such  that
	$$\lim_{n\to\infty}\frac{h(F_n)}{|F_n|}=\lambda$$
	for all F\o lner sequences $\{F_n\}_{n\ge1}$  of $G$.
\end{lem}

Let $(X,d)$ be a compact metric space.  Given   a finite open cover $\mathcal{U}$ of $X$, the \emph{diameter} of $\mathcal{U}$,   denoted by $\rm{diam} \mathcal{U}$,  is   the maximum of the diameters of its elements  w.r.t. $d$.  The \emph{Lebesgue number} of $\mathcal{U}$,   denoted by $\Leb(\UU)$,  is  the largest  positive number $\delta >0$ such that  each  open ball  $B_d(x,\delta)$ is contained in  some  element of $\mathcal{U}$. The \emph{join} of   two finite open covers $\mathcal{U}, \mathcal{V}$ of $X$ is    defined by   $\mathcal{U}\vee \mathcal{V}:=\{U\cap V: U\in \mathcal{U}, V\in \mathcal{V}\}$.

Given  a finite open cover  $\mathcal{U}$ of $X$ and $F\in \FF(G)$, we  put
\begin{align}
	Q_F(G,f,\omega, \mathcal{U})
	=\inf\{\sum_{A\in \eta}\sup_{x\in A}e^{S_Ff(\omega,x)}: \eta\mbox{ is a } \mbox{subcover of} \vee_{g\in F}(T_{g,\omega})^{-1}\nonumber \mathcal{U}\}.
\end{align}
One has  $Q_F(G,f,\omega, \mathcal{U})$ is measurable in $\omega$ by the proof of \cite[Proposition 1.6]{kif01}. Let $\{F_n\}$ be a F\o lner sequence of $G$.  We define
$$	Q(G,f,\mathcal{U})=\lim_{n\to \infty}  \frac{1}{|F_n|}\int   \log Q_{F_n}(G,f,\omega, \mathcal{U})d\P(\omega).$$

\begin{prop}
Let $T=(T_{\omega})$ be a RDS over the measure-preserving system $(\Omega,\mathcal{F},\mathbb{P},\theta)$. Put 
$$h: F\in \FF(G) \mapsto  \int   \log Q_{F}(G,f,\omega, \mathcal{U})d\P(\omega).$$
Then  $h$  is a monotone, $G$-invariant and  sub-additive function.
 
Consequently,  $Q(G,f, \mathcal{U})$ is independent of the choice of F\o lner sequence of $G$.
\end{prop}

\begin{proof}
Let  $f  \in L^1(\Omega, C(X))$ with $f\geq 0$. Clearly, $h$ is  monotone.

Fix $E \in \FF(G)$ and $h\in G$.  It follows from the cocycle property that
$\vee_{g\in Eh}(T_{g,\omega})^{-1}\nonumber \mathcal{U}=(T_{h,\omega})^{-1}\vee_{g\in E}(T_{g,h\omega})^{-1}\nonumber \mathcal{U}.$
Suppose that $\eta$ is a subcover of $\vee_{g\in E}(T_{g,h\omega})^{-1}\nonumber \mathcal{U}$. Then $(T_{h,\omega})^{-1}\eta$ is a subcover of  $\vee_{g\in Eh}(T_{g,\omega})^{-1}\nonumber \mathcal{U}$.  Thus,
\begin{align*}
Q_{Eh}(G,f,\omega, \mathcal{U}) \leq \sum_{A\in \eta}\sup_{x\in (T_{h,\omega})^{-1} A}e^{S_{Eh}f(\omega,x)}&\leq  \sum_{A\in \eta}\sup_{x\in A}e^{S_Ef(h\omega,x)}.
\end{align*}
This implies that  $Q_{Eh}(G,f,\omega, \mathcal{U})\leq Q_{E}(G,f,h\omega, \mathcal{U})$. By the $G$-invariance of $\mathbb{P}$, we have $h(Eh)\leq h(E)$. Hence, $h(Eh)= h(E)$ for all $E\in \FF(G)$ and $h\in G$. This shows that $h$ is  $G$-invariant.

Now choose two disjoint $E,F \in \FF(G)$.  Let  $\gamma >0$ and  choose  a subcover $\eta_1$ of $\vee_{g\in E}(T_{g,\omega})^{-1}\nonumber \mathcal{U}$  and  a subcover $\eta_2$ of $\vee_{g\in F}(T_{g,\omega})^{-1}\nonumber \mathcal{U}$ such that  $$ \sum_{A\in \eta_1}\sup_{x\in A}e^{S_Ff(\omega,x)}<	Q_E(G,f,\omega, \mathcal{U})+\frac{\gamma}{2}$$ and $ \sum_{A\in \eta_2}\sup_{x\in A}e^{S_Ff(\omega,x)}<	Q_F(G,f,\omega, \mathcal{U})+\frac{\gamma}{2}.$ Then $\eta_1\vee \eta_2$ is a subcover of  $\vee_{g\in E\cup F}(T_{g,\omega})^{-1}\nonumber \mathcal{U}$. Hence, we have
\begin{align*}
	Q_{E\cup F}(G,f,\omega, \mathcal{U})& \leq \sum_{A \cap B\in \eta_1\vee \eta_2}\sup_{x \in  A\cap B}e^{S_{E\cup F}f(\omega,x)}\\
	 &\leq (Q_E(G,f,\omega, \mathcal{U})+\frac{\gamma}{2}) (Q_F(G,f,\omega, \mathcal{U})+\frac{\gamma}{2}). 
\end{align*}
This implies that $h(E\cup F) \leq h(E) + h(F)$. 

Therefore, by Lemma \ref{lem 5.1}, the limit $	\lim_{n\to\infty}\frac{h(F_n)}{|F_n|}$ is independent of the choice of  F\o lner sequence of $G$.  We verify this fact for any $f\in L^{1}(\Omega, C(X))$ by reducing to the non-negative case.  Let $\tilde{f}(\omega,x)=f(\omega,x)-\min_{x\in X}f(\omega,x)\geq 0$. Then
$$	Q_F(G,\tilde{f},\omega, \mathcal{U})= e^{-\min_{x\in X}f(\omega,x) \cdot |F|}	Q_F(G,f,\omega, \mathcal{U}).$$
Then for any  F\o lner sequence of $G$, we have
\begin{align*}
&\lim_{n\to \infty}  \frac{1}{|F_n|}\int   \log Q_{F_n}(G,\tilde{f},\omega, \mathcal{U})d\P(\omega)\\
=&\lim_{n\to \infty}  \frac{1}{|F_n|}\int   \log Q_{F_n}(G,f,\omega, \mathcal{U})d\P(\omega)- \int_\Omega \min_{x\in X}f(\omega,x)d\mathbb{P}(\omega),
\end{align*}
where the second term on the right-hand side of the equality is a  constant.

To sum up,  $Q(G,f, \mathcal{U})$ is independent of the choice of F\o lner sequence of $G$.
\end{proof}

\begin{prop}\label{prop 5.3}
Let $T=(T_{\omega})$ be a RDS over the measure-preserving system $(\Omega,\mathcal{F},\mathbb{P},\theta)$. Let $f\in L^1(\Omega, C(X))$. Then for  every  F\o lner sequence  $\{F_n\}$ of  $G$,
\begin{align*}
\lim_{\epsilon \to 0} P(G,f,d,\{F_n\},\epsilon)=\sup_{\UU \in \mathcal{C}_X^o}Q(G,f,\mathcal{U}).
\end{align*}
Consequently, $P(G,f)$ is well-defined.
\end{prop}
\begin{proof}
Fix a F\o lner sequence  $\{F_n\}$ of $G$. Let  $\epsilon >0$ and $F \in \FF(G)$. Let  $\mathcal{U}$  be a finite open cover of $X$ with $\rm{diam} \mathcal{U}\leq \epsilon$. Let $E$ be an  $(\omega, F,\epsilon)$-separated set of 
$X$. Then  for any subcover  $\eta$ of  $\vee_{g \in F}(T_{g,\omega})^{-1} \mathcal{U}$,  each element of $\eta$ at most contains one  element $x\in E$. This implies that  $P_F(G,f,d,\omega,\epsilon)\leq  Q_F(G,f,\omega, \mathcal{U})$, and hence
\begin{align}\label{inequ 5.1}
 P(G,f,d,\{F_n\},\epsilon) \leq \inf_{\diam \UU \leq \epsilon}Q(G,f,\mathcal{U}).
\end{align}

Let  $S$ be a finite $\frac{\epsilon}{4}$-net of  $X$.  Let $\VV=\{B_d(x,\frac{\epsilon}{2}): x\in S\}$. Then  $\VV$ is   a finite open cover  of $X$  with $\rm{diam} \mathcal{V}\leq\epsilon$ and ${\rm Leb} (\mathcal{V})\geq \frac{\epsilon}{4}$. Let $E$ be an $(\omega, F,\frac{\epsilon}{4})$-separated set of  $X$ with the largest cardinality,  and  let $\eta$  be a  subcover of  $\vee_{g\in F}(T_{g,\omega})^{-1} \mathcal{V}$. Then  for any $x\in E$,  the open ball $B_F^{\omega}(x,\frac{\epsilon}{4})$ is contained in  some elements  of $\eta$. We fix an element $A_x\in \eta$  for every $x\in E$. Let $\eta_1 \subset \eta$  be a corresponding  subcover containing these open balls. Then
\begin{align*}
\begin{split}
Q_F(G,f,\omega, \mathcal{V})&\leq  \sum_{A_x\in \eta_1,x\in E}\sup_{y\in A_x} e^{S_Ff(\omega,y)}\\
&\leq  \sum_{x\in E}e^{S_Ff(\omega,x)+\sum_{g \in F}\gamma_{\epsilon}(g\omega)},
\end{split}
\end{align*}
where $\gamma_{\epsilon}(\omega):=\sup\{|f(\omega,x)-f(\omega,y)|: x,y\in X, d(x,y)<2\epsilon\}$. This implies that
\begin{align}\label{equ 2.5}
Q_F(G,f,\omega, \mathcal{V})\leq  e^{\sum_{g\in F}\gamma_{\epsilon}(g\omega)} P_F(G,f,d,\omega,\frac{\epsilon}{4}).
\end{align}

We show that  $\gamma_{\epsilon}(\omega)\in L^{1}(\Omega, \mathcal{F},\mathbb{P})$.
Define  a map $\hat{f}:\Omega \times (X\times X)\rightarrow \mathbb{R}_{\geq 0}$  given by $\hat{f}(\omega, x,y)$= $|f(\omega,x)-f(\omega,y)|$. Since $f\in L^1(\Omega, C(X)),$ then $\hat{f}(\omega, \cdot,\cdot)$   is continuous on $X\times X$ for any fixed $\omega$,  and $\hat{f}(\cdot,x,y)$ is measurable in  $\omega$  for any fixed $(x,y)\in X\times X$.   Choose two  countable dense subsets $S_1,S_2$ of $X$.  Then  for any fixed $\omega \in \Omega$,
$$\sup_{(x,y)\in X\times X,\atop d(x,y)<2\epsilon}\hat{f}(\omega,x,y)=\sup_{(x,y)\in S_1\times S_2,\atop d(x,y)<2\epsilon}\hat{f}(\omega,x,y).$$
Let $r>0$. Then 
\begin{align*}
\{\omega\in\Omega:\gamma_{\epsilon}(\omega)\leq r\}=&\{\omega\in\Omega:\sup_{(x,y)\in X\times X,\atop d(x,y)<2\epsilon}\hat{f}(\omega,x,y)\leq r\}\\
=&\{\omega\in\Omega:\sup_{(x,y)\in S_1\times S_2,\atop d(x,y)<2\epsilon}\hat{f}(\omega,x,y)\leq r\}\\
=&\mathop{\cap}_{(x,y)\in S_1\times S_2, \atop d(x,y)<2\epsilon}\{\omega\in\Omega:\hat{f}(\omega,x,y)\leq r\}\in \mathcal{F},
\end{align*}
which implies that $\gamma_{\epsilon}(\omega)$ is  measurable in $\omega$. Notice that  $\gamma_{\epsilon}(\omega)\leq2||f(\omega)||_{\infty}$ and $f\in L^1(\Omega, C(X))$. Then  $\gamma_{\epsilon}\in L^{1}(\Omega, \mathcal{F},\mathbb{P})$.

Hence, by (\ref{equ 2.5}) we have
\begin{align}\label{inequ 5.3}
\inf_{\rm{diam} \mathcal{V} \leq \epsilon}  Q(G,f,\mathcal{V})
\leq  P(G,f,d,\{F_n\},\frac{\epsilon}{4})
+\int \gamma_{\epsilon}(\omega)d\P(\omega).
\end{align}
Since $\gamma_{\epsilon}(\omega) \to 0$ as $\epsilon \to0$ for all  $\omega$,  then $$\lim_{\epsilon \to 0}\int \gamma_{\epsilon}(\omega) d\P(\omega)=\lim_{n \to \infty}\int \gamma_{\frac{1}{n}}(\omega) d\P(\omega)=0$$ by Levi's monotone convergence theorem.  Notice that $$\sup_{\UU \in \mathcal{C}_X^o}Q(G,f,\mathcal{U})=\lim_{\epsilon \to 0} \inf_{\rm{diam} \mathcal{V} \leq \epsilon}  Q(G,f,\mathcal{V}).$$  Combining this fact with (\ref{inequ 5.1}) and (\ref{inequ 5.3}), we get  $$\lim_{\epsilon \to 0} P(G,f,d,\{F_n\},\epsilon)=\sup_{\UU \in \mathcal{C}_X^o}Q(G,f,\mathcal{U}).$$
\end{proof}

\section*{Acknowledgement} 
 
The authors would like to thank Prof. Paulo Varandas  for  many useful discussions during the preparation of the manuscript. The first author was  supported by  the China Postdoctoral Science Foundation (No. 2024M763856) and  the Postdoctoral Fellowship Program of CPSF  (No. GZC20252040). The  second author was supported by the
National Natural Science Foundation of China (Nos.12471184 and 12071222).
The  third author   was  supported by the
National Natural Science Foundation of China (No. 11971236) and Qinglan project. 




\end{document}